\documentclass{amsart}[11pt]

\usepackage{amssymb}
\usepackage{amsmath}
\usepackage{amsfonts}
\usepackage{yfonts}
\usepackage{amsthm}	
\usepackage{bbm}
\usepackage{dsfont}
\usepackage[shortlabels]{enumitem}
\usepackage{multicol}
\usepackage{ragged2e}
\usepackage{wrapfig}
\usepackage{geometry}
\usepackage{graphicx}
\usepackage{tikz}
\usepackage{hyperref}
\hypersetup{citecolor=red}

\newcommand{\R}{\mathbb{R}} %reals
\newcommand{\Q}{\mathbb{Q}} %rationals
\newcommand{\C}{\mathbb{C}} %complex numbers
\newcommand{\N}{\mathbb{N}} %natural numbers
\newcommand{\D}{\mathbb{D}} %unit disc 
\newcommand{\T}{\mathbb{T}} %unit cirlce
\newcommand{\BB}{\mathcal{B}} %Bloch space
\newcommand{\sbs}{\subset} %subset
\newcommand{\ind}{\mathbbm{1}} %indicator function of a set
\newcommand{\eps}{\varepsilon}

\newcommand{\intT}[1]{\int_{\T} #1 \,dm(\zeta)} %int on unit circle 
\newcommand{\no}[1]{\Vert #1 \Vert} %norm of vec
\newcommand{\Ra}{\Rightarrow}
\newcommand{\ra}{\rightarrow} 
\newcommand{\fa}{\forall}

\renewcommand{\it}[1]{\textit{#1}}
\renewcommand{\bf}[1]{\textbf{#1}}

\makeatletter
\renewcommand\subsection{\@startsection{subsection}{2}%
\z@{-.5\linespacing\@plus-.7\linespacing}{.5\linespacing}%
{\normalfont\scshape}}
\renewcommand\subsubsection{\@startsection{subsubsection}{3}%
\z@{.5\linespacing\@plus.7\linespacing}{-.5em}%
{\normalfont\scshape}}
\makeatother

\newtheorem{theorem}{Theorem}[section]
\newtheorem{lemma}[theorem]{Lemma}
\newtheorem{remark}[theorem]{Remark}
\newtheorem{corollary}[theorem]{Corollary}
\newtheorem{proposition}[theorem]{Proposition}

\theoremstyle{definition}

\usepackage[backend=biber,style=numeric,sorting=nty,giveninits=true,isbn=false]{biblatex} 
\nocite{*}

\title{Shift Invariant Subspaces Of Large Index In The Bloch Space}
\author{Nikiforos Biehler}

\begin{document}

%------------------------------------------------------------------------------
%
% Title, Abstract
%
%------------------------------------------------------------------------------

\begin{abstract}
We consider the shift operator $M_z$, defined on the Bloch space and the little Bloch space and we study the corresponding lattice of invariant subspaces. The index of a closed invariant subspace $E$ is defined as $\text{ind}(E) = \dim(E/M_zE)$.  We construct closed, shift invariant subspaces in the Bloch space that can have index as large as the cardinality of the unit interval $[0,1]$. Next we focus on the little Bloch space, providing a construction of closed, shift invariant subspaces that have arbitrary large index. Finally we establish several results on the index for the weak-star topology of a Banach space and prove a stability theorem for the index when passing from (norm closed) invariant subspaces of a Banach space to their weak-star closure in its second dual. This is then applied to prove the existence of weak-star closed invariant subspaces of arbitrary index in the Bloch space.

\end{abstract}

\maketitle

\section{Introduction \& main results}

Consider the unit disc $\D = \{ z \in \C \, : \, |z| < 1 \} $ of the complex plane. The \textit{Bloch space} $\BB$ is defined as the set of the functions $f$, analytic in $\D$, that satisfy $\sup_{|z|<1}(1-|z|^2)|f'(z)| < +\infty$. The quantity 

\begin{equation*}
\no{f}_{\BB} = |f(0)| + \sup_{|z|<1}(1-|z|^2)|f'(z)|
\end{equation*}

\noindent is a norm on the Bloch space, which makes it into a Banach space. The closure of analytic polynomials with respect to that norm is a subspace of the Bloch space, called the \textit{little Bloch space}, and is denoted by $\BB_0$. An equivalent way of defining the little Bloch space, is as the subspace of functions in the Bloch space that satisfy $\lim_{|z|\ra 1^-}(1-|z|^2)|f'(z)| = 0$. Functions in the Bloch space enjoy several nice properties. For a function $f$ in the Bloch space, the quantity $\sup_{|z|<1}(1-|z|^2)|f'(z)|$ is Möbius invariant, meaning it remains unchanged after composing $f$ on the right by any automorphism of the unit disc. Another well known fact is that an analytic function belongs to the Bloch space if and only if it is Lipschitz with respect to the hyperbolic metric on the unit disc. The Bloch space has been thoroughly studied (in \cite{Pommerenke1974},\cite{Pommerenke1992} for example) as it is linked with many topics in analytic function theory. 

\vspace{3mm}

\noindent We consider the operator of multiplication $M_z : \BB \ra \BB$, $M_zf(z) = zf(z)$, also called the Shift operator, and we are interested in studying the lattice of closed invariant subspaces of $M_z$. Let $E$ be a closed shift invariant subspace.

%------------
\par\noindent\rule{2.5cm}{0.4pt}

\noindent \textit{2020 Mathematics Subject Classification:} 30H30, 30B10, 47A15, 47B91. 
\\
\noindent \textit{Keywords and phrases:} Bloch space, little Bloch space, invariant subspaces, shift operator, index, lacunary Taylor series.
%------------

\newpage

\noindent It follows from general properties of the shift operator that $zE$ is closed (see Lemma 1.2), allowing us to define the \textit{index} of $E$ to be the quantity

\begin{equation*}
\text{ind}(E) := \dim(E/zE).
\end{equation*}

\noindent Our goal is to show that, for the spaces under consideration, there exist invariant subspaces for which the index can be as large as possible, that is to say, as large as the space permits it to be. 

\vspace{3mm}

\noindent Our motivation comes from a series of papers, starting of course from the celebrated Beurling Theorem, which characterizes the invariant subspaces of the shift operator in the classical Hardy space $H^2$. A fact following from that characterisation is that every invariant subspace has the so called \textit{codimension one} (also called \textit{index one}) property, i.e. every non-trivial invariant subspace has index equal to one. A classic reference for properties of the index, in this context, and the codimension one property is Richter's article \cite{Richter1987}.

\vspace{3mm}

\noindent In a 1985 paper C. Apostol, H. Bercovici, C. Foias and C. Pearcy (\cite{Apostol1985}) proved that what was previously true for the Hardy space, is no longer true for the classical Bergman space $A^2$ of functions analytic in the unit disc and square integrable with respect to planar Lebesgue measure. In particular they proved the existence of invariant subspaces $E_k$ such that $\text{ind}(E_k) = k$, for $k =1,2, \ldots, +\infty$. Numerous results have been published since then, in hope to understand the lattice of invariant subspaces of the shift operator. In one direction, we have several results proving the codimension one property, such as in the classical Dirichlet space (\cite{Richter1988}) of analytic functions in the unit disc whose derivative belongs to $A^2$, or in the space $\ell^1_{A}$ of Taylor series in the unit disc with summable coefficients (\cite{Richter1987}). In the other direction we have plenty of constructions, utilising different properties of each space, to prove the existence of invariant subspaces of arbitrary index. The first concrete example of an invariant subspace of index 2 in $A^2$ has been given by H. Hedenmalm in \cite{Hedenmalm1993}, using results on sampling and interpolation in the Bergman space. Later lacunary series were used by A. Borichev for a wide range of spaces including the classical Bergman spaces, a variety of mixed norm spaces, growth spaces and some weighted sequence spaces (\cite{Borichev1998}). In \cite{Abakumov2002} E. Abakumov and A. Borichev proved the existence of invariant subspaces of arbitrary index for a variety of weighted sequence spaces using solution sets of convolution equations. In particular they show that the space $\ell^p_{A}$, of Taylor series in the unit disc, with $p$-summable Taylor coefficients, contains invariant subspaces of arbitrary index as long as $p>2$. The case $1<p<2$ is still an open problem. Special families of inner functions  have been used for $H^{\infty}$ (\cite{Borichev1998},\cite{Niwa2003}). 

\vspace{3mm}

In the recent years the Bloch space has attracted a lot of attention. In a recent paper (\cite{Limani2023}), A. Limani and A. Nicolau answer several open questions in the Bloch space, related to invariant subspaces and cyclicity. In particular, they prove a Beurling-type theorem for singly generated, weak-star invariant subspaces in the Bloch space. This gives more motivation to study the index of the shift invariant subspaces in the Bloch space.

\vspace{3mm}

In this paper, we exploit ideas developed by A. Borichev in \cite{Borichev1998} and construct lacunary series, with almost maximal growth, in order to prove the existence of closed, shift invariant subspaces of arbitrary large index in each of the spaces $\BB$ and $\BB_0$ as well as weak-star closed invariant subspaces in $\BB$, which under this topology is not a Banach space. The maximal growth for Bloch functions can be obtained by integrating the derivative of a function and using the definition of the norm. 

\newpage

\noindent In particular for every $f \in \BB$ we have:

\begin{equation}
|f(z)| \leq \frac{1}{2} \no{f}_{\BB} \log \frac{1 + |z|}{1- |z|} \lesssim  \no{f}_{\BB} \log \frac{1}{1- |z|}.
\end{equation}

\noindent A more accurate result concerning the growth of Bloch functions comes from Makarov's Law of Iterated Logarithm (\cite{Makarov1985}). Inequality (1) also means that integrating a Bloch function, results in a function in the little Bloch space, which will be used with little mention throughout the text. 

\vspace{3mm}

\noindent By \textit{lacunary series} we will always mean some function of the form:

\begin{gather*}
f(z) = \sum_{n=0}^{\infty} a_n z^{s_n} \,\,, \,z \in \D \, , 
\end{gather*}

\noindent where $\{a_n\}_{n=0}^{\infty} \sbs \C$ and $\{s_n\}_{n=0}^{\infty} \sbs \N$, and such that the exponents $s_n$ grow sufficiently fast. The rate at which the sequence $s_n$ should grow is described by the Ostrowski-Hadamard gap theorem (see also Fabry's gap theorem), which asserts that if $\frac{s_{n+1}}{s_n} \geq q > 1$ for some $q>1$ and all $n \in \N$, and if the Taylor coefficients are such that the radius of convergence of the above series is equal to 1, then that function has the unit circle as a natural boundary. Such sequences $\{s_n\} \sbs \N$ will simply be called lacunary sequences. Even though series of this type are notoriously badly behaved, they can be useful to construct functions with one's desired properties. Lacunary functions are neatly characterised in the Bloch space:

\begin{proposition}
Let $f(z) = \sum_{n=0}^{\infty} a_n z^{s_n}$ be a lacunary series. Then $f$ belongs to the Bloch space (little Bloch space resp.) if and only if $ (a_n) $ is bounded ( $a_n \ra 0$ resp.)
\end{proposition}
 
\noindent Proof of this can be found in \cite{Hedenmalm2000}, Theorem 1.14. A key tool in proving a given invariant subspace has the desired index is the following lemma on summation of indices.

\begin{lemma}
Let $X$ be a Banach space of analytic functions on the unit disc, satisfying the division property(\cite{Richter1987}), i.e. 

\begin{enumerate}[label=(\roman*)]

\item $X$ is a Banach space contained in $Hol(\D)$, the space of analytic functions on $\D$.

\item Evaluation functionals $k_{\lambda}$ are bounded on $X$, for all $\lambda \in \D$.

\item $zf$ belongs to $X$ whenever $f \in X$.

\item If $f \in X$ and $f(\lambda) = 0$, then there exists a function $g\in X$ such that $(z - \lambda)g(z) = f(z)$.

\end{enumerate}

Let $N \in \{1,2, \ldots, +\infty \}$, and for each $0 \leq n < N$ let $E_n \sbs X$ be an invariant subspaces of index one, and $f_n \in E_n$ with $f_n(0) \neq 0$. Suppose moreover that for each $0 \leq n < N$ there exists $c_n > 0$ such that

\begin{equation}
c_n|g(0)| \leq \no{g + h }_X \, ,  \quad g \in E_n \, , h \in \bigvee_{k \neq n}E_k.
\end{equation}

Then $\text{ind}(\bigvee_{0 \leq n < N} E_n) = N$.

\end{lemma}

\noindent A proof of the above lemma is given in \cite{Borichev1998}. In the original paper of S. Richter \cite{Richter1987} it is proven that when the shift operator $M_z$ is defined on such a space $X$, then $M_z$ is bounded from above and below, meaning also that $zE$ is closed whenever $E$ is closed. As a consequence the index is well defined for every closed invariant subspace. In our case the spaces $E_n$ will always be cyclic invariant subspaces $[f_n] := \overline{\text{span}} \{ pf_n \, : \, p \,\, \text{polynomial} \} $ with $f_n(0) \neq 0$, and it will be enough to check condition (2) for polynomial multiples of the generating functions. A proof of the fact that the Bloch space verifies the requirements of Lemma 1.2. will be provided in the Section 2. 

\vspace{3mm}

 With all the preliminary information established, we present the results contained in this article.

\begin{theorem}
For every $N \in \{ 1 , 2, \ldots , + \infty \}$ there exists an invariant subspace $E_N \sbs \BB$ such that $\text{ind}(E_N) = N $.
\end{theorem}

\noindent Since the Bloch space is non-separable with that norm we are able to produce an example of an invariant subspace with uncountable index.

\begin{theorem}
There exists an invariant subspace $E$ of $\BB$ such that $\text{ind}(E) = \text{card}([0,1])$.
\end{theorem}

\noindent The above two theorems utilise elementary properties of the Bloch space and its norm, and the vectors generating the cyclic subspaces are constructed inductively. If one wants to pass from the Bloch space to its ``little-oh'' analogue, one will be forced to use different techniques. Since functions in the little Bloch space satisfy $(1-|z|^2)|f'(z)| \ra 0$ as $|z| \ra 1^-$, the norm on the space is unable to capture the behaviour of the function by looking close to the boundary. This impediment proves fatal to the argument used in Theorems 1.2 and 1.3, but gave us motivation to develop a new argument, which requires the construction of lacunary functions with several special properties. The growth of the functions, the most crucial part of their behaviour, is studied with the aid of a classical theorem of R. Salem and A. Zygmund about the distribution function of trigonometric series. In this approach we construct functions which are almost maximally large on sufficiently massive subsets of the unit disc, and use an iterative argument to prove good bounds on the $L^p$ means of polynomials on the aforementioned sets.

\begin{theorem}
For every $N \in \{ 1 , 2, \ldots , + \infty \}$ there exists an invariant subspace $E_N \sbs \BB_0$ such that $\text{ind}(E_N) = N $.
\end{theorem}

\noindent Finally, we turn to a more abstract setting and extend several results from S. Richter's paper \cite{Richter1987} for the weak-star topology. In particular, let $X_0, X, Y$ be Banach spaces of analytic functions satisfying the division property, and that satisfy the following dualities: $X_0^{\ast} \cong Y$ and  $Y^{\ast} \cong X$. The space $X_0$ can always be continuously embedded into $X$. Given an invariant subspace $E \sbs X_0$, write $\overline{E}^{w^{\ast}}$ for its weak-star closure in $X$. Under some natural assumptions on the spaces we prove the following theorem: 

\begin{theorem}
Let $E$ be an invariant subspace of $X_0$ such that $\text{ind}(E)= N$ for some $N \in \{1, 2, \ldots ,+\infty \}$. Then the subspace $\overline{E}^{w^{\ast}} \sbs X$ is invariant, weak-star closed and $\text{ind}(\overline{E}^{w^{\ast}}) = N$.
\end{theorem}

We may then equip the Bloch space with the weak-star topology, inherited from its pre-dual, which can be identified as the Bergman space $A^1(\D)$. This weaker topology makes the Bloch space separable, and can be viewed as a more natural choice for topology when studying certain problems, such as existence of cyclic vectors of the shift operator. With this set-up, it makes sense to look for invariant subspaces that are weak-star closed. Since the space becomes separable with the weak-star topology we cannot expect the index to be as large as we demonstrate in Theorem 1.3, but at most countable. Theorem 1.6 can be combined with Theorem 1.4 to obtain weak-star closed, invariant subspaces of arbitrary large, index for the Bloch space.

\begin{theorem}
For every $N \in \{ 1 , 2, \ldots , + \infty \}$ there exists a weak-star closed invariant subspace $E_N \sbs \BB$ such that $\text{ind}(E_N) = N $.
\end{theorem}

\noindent Theorem 1.7 provides an interesting antithesis to a phenomenon previously known in $H^{\infty}$. There, the invariant subspaces behave quite differently when we pass from a strong to a weaker topology. On one hand, there are plenty of examples demonstrating that norm-closed invariant subspaces may have arbitrary large index. On the other hand, once the space is equipped with the weak-star topology, it a Beurling-type theorem will hold, as in the classical case. Theorem 1.7 demonstrates that it is no longer the case in the Bloch space.

\begin{remark}
    In each one of the Theorems 1.3, 1.5, 1.7 it is actually proven that there is a sequence of functions $f_n$, $1 \leq n < \infty$ in the appropriate space (either $\BB$ or $\BB_0$) such that for every $1 \leq n_1 < \cdots < n_N < \infty$, the invariant subspace $\bigvee_{1 \leq k \leq N}[f_{n_k}]$ has index equal to $N$. 
\end{remark}

The paper is organized as follows. Section 2 is dedicated to the proof of Theorems 1.3 and 1.4. In section 3 we provide the proof as well as the necessary backgroung for Theorem 1.5. Finally section 4 is left to prove Theorem 1.6 and Corollary 1.7, as well as to extend the required results from Richter's article.

%SECTION  2

\section{Invariant subspaces in the Bloch space}

We begin by verifying that the Bloch space satisfies the hypotheses of Lemma 1.2, i.e, that it is indeed a Banach space of analytic functions satisfying the division property. The case of the little Bloch space is almost identical; one simply needs to verify that dividing out the zero of a given function in the little Bloch space, results in a function which is still in the little Bloch space.  Property \it{(i)} is obvious, and inequality (1) guarantees \it{(ii)} and \it{(iii)} are satisfied. It remains to check \it{(iv)}, a.k.a the division property. Let $E_0 = \{ f \in \BB \, | \, f(0) = 0 \} $, and $R_0 : E_0 \ra \BB $ the operator that maps $f \mapsto f(z) / z  $. We will prove that $R_0$ is bounded, and hence well-defined as well. The general case $\lambda \in \D$ follows from the Möbius invariance of the Bloch space.

Consider $f \in \BB$ with $f(0) = 0$ and fix $0 < s < 1 $. We can find an analytic function $g \in Hol(\D)$ such that $f(z) = zg(z)$. We have:

\begin{gather*}
\Vert R_0(f) \Vert_{\BB} = |g(0)| + \sup_{|z|<1} (1-|z|^2)|g'(z)|.
\end{gather*}

Obviously we have that $g(0) = f'(0)$ so $|g(0)| \leq \Vert f \Vert_{\BB}$. For the rest we may write :

\begin{gather*}
\sup_{|z|<1} (1-|z|^2)|g'(z)| \leq \max \bigg\{\sup_{|z|\leq s} (1-|z|^2)|g'(z)| , \sup_{s<|z|<1} (1-|z|^2)|g'(z)|  \bigg\}
\end{gather*}

Let $0< \eps < 1-s$. and define $\gamma$ to be the anti-clockwise oriented circle centered at the origin and of radius $s + \eps $. Then, when $|z|< s + \eps$ we get by Cauchy's formula:

\begin{gather*}
g'(z) = \frac{1}{2\pi i} \int_{\gamma} \frac{g'(\zeta)}{\zeta - z} \,d\zeta =\frac{1}{2\pi i} \int_{\gamma} \frac{f'(\zeta)-g(\zeta)}{\zeta(\zeta - z)} \,d\zeta = \\
= \frac{1}{2\pi i}\int_{\gamma}\frac{f'(\zeta)}{\zeta(\zeta - z)} \,d\zeta -\frac{1}{2\pi i} \int_{\gamma} \frac{\zeta g(\zeta)}{\zeta^2(\zeta - z)} \,d\zeta = \\
= \frac{1}{2\pi i}\int_{\gamma} \frac{f'(\zeta)(1-|\zeta|^2)}{\zeta(\zeta - z)(1-|\zeta|^2)} \,d\zeta - \frac{1}{2\pi i}\int_{\gamma} \frac{\zeta g(\zeta)(1-|\zeta|^2)}{\zeta^2(\zeta - z)(1-|\zeta|^2)} \,d\zeta.
\end{gather*}

From this we get the estimate:

\begin{align*}
\sup_{|z| \leq s}(1-|z|^2)|g'(z)| & \leq \sup_{|z| \leq s+\eps}(1-|z|^2)|f'(z)|\cdot \frac{1}{s}\frac{1}{1-(s+\eps)^2}  \\
& + \quad \, \sup_{|z| \leq s+\eps}|(1-|z|^2)|f(z)|\cdot \frac{1}{s(s+\eps)}\frac{1}{1-(s+\eps)^2}.
\end{align*}

This is true for every $\eps >0$ small enough, hence the above inequality, along with the fact that the integration operator is bounded on the Bloch space (inequality (1)), we get that there is a constant $C(s)>0$, independent of $f$ such that $\sup_{|z|<s} (1-|z|^2)|g'(z)| \leq C(s) \no{f}_{\BB}$. On the other hand : 

\begin{gather*}
\sup_{s <|z|<1} (1-|z|^2)|g'(z)| \leq \frac{1}{s} \sup_{s <|z|<1} (1-|z|^2)|zg'(z)| = \frac{1}{s} \sup_{s <|z|<1} (1-|z|^2)|f'(z) - g(z)| \leq \\
\leq \frac{1}{s} \sup_{s <|z|<1} (1-|z|^2)|f'(z)| + \frac{1}{s} \sup_{\eps <|z|<1} (1-|z|^2)|g(z)| \leq \\
\leq \frac{1}{s} \sup_{s <|z|<1} (1-|z|^2)|f'(z)| + \frac{1}{s^2} \sup_{s <|z|<1} (1-|z|^2)|f(z)|.
\end{gather*}

Similarly we obtain a constant $C'(s)>0$ such that $\sup_{s< |z|<1} (1-|z|^2)|g'(z)| \leq C'(s) \no{f}_{\BB}$. Overall $\no{g}_{\BB} \leq \max\{1, C(s), C'(s)\} \no{f}_{\BB}$, and the proof is finished. 

\vspace{3mm}

We may pass to the proof of Theorems 1.3 and 1.4. For every $n\in \N$ we introduce the auxiliary function $U_n(r) = (1-r^2)nr^{n-1}$ , $0 \leq r < 1$. Note that $\sup_{0<r<1}U_n(r) = \no{z^n}_{\BB}$. We also consider the sequence of radii $r_n = (1-\frac{1}{n})^{\frac{1}{2}}$. The radius $r_n$ is sufficiently close to the maximizing point of the function $U_n$, and we see that $U_n(r_n) \ra \frac{1}{\sqrt{e}}$. We begin by stating a lemma that describes the construction of several lacunary sequences with some additional properties. 

\begin{lemma}
There exists a sequence $\{ s(n,i)\}_{1 \leq i \leq n < \infty } \sbs \N$ such that:

\begin{enumerate}[label=(\roman*)]
\item $1 < s(1,1) < s(2,1) < s(2,2) < s(3,1) < s(3,2) < s(3,3) < s(4,1) < s(4,2) < \cdots$

\vspace{1mm}

\item For every $1 \leq i \leq n$, we have $\frac{s(n+1,i)}{s(n,i)} \geq 2$, for all $n \in \N$.

\vspace{1mm}

\item For every $(n,i) \neq (n',i')$ we have $U_{s(n,i)}(r_{s(n',i')}) < \frac{1}{2^{n+n'}}$.
\end{enumerate}

\end{lemma}

\noindent The proof of this lemma will be given at the end of this section. We may proceed to prove Theorem 1.3. 

%%%%%%%%%%%%%%%%%%%%%%%%%%%%%%%%%%%%%%%%%%%
%
%
% PROOF OF THEOREM 1.3
%
%
%%%%%%%%%%%%%%%%%%%%%%%%%%%%%%%%%%%%%%%%%%%

\textit{Proof of Theorem 1.3.}

\vspace{2mm}

 We consider a sequence  $\{ s(n,i)\}_{1 \leq i \leq n < \infty} \sbs \N$, as constructed using Lemma 2.1. We define a collection of functions as follows:

\begin{gather}
f_i(z) = 1 + \sum_{n=i}^{\infty}z^{s(n,i)} \,\,, z\in \D, \, \,1 \leq i < \infty.
\end{gather}

\noindent Note that condition \textit{(ii)} of Lemma 2.1 states that for each $i$, the function $f_i$ is lacunary. Moreover, Proposition 1.1 guarantees that these functions all belong to the Bloch space as their Taylor coefficients are bounded. For each $N \in \N$ we define the invariant subspace $E_N = \bigvee_{1 \leq i \leq N} [f_i] \sbs \BB$. For $N = \infty$ we define $E_{\infty} =\bigvee_{1 \leq i < \infty} [f_i] \sbs \BB $ and our goal is to show that $\text{ind}(E_N) = N$. According to Lemma 1.2, it suffices to show that for every $ 0 \leq M < \infty $ and for every $ 1 \leq i_0 < \infty$ there exists some $c_{i_0} > 0$ such that for every $1 \leq i_1 < i_2 < \cdots < i_M $, with $i_j \neq i_0$ for $j \neq 0$, and for every polynomials $p_0, p_1, \ldots, p_M$, we have:

\begin{gather}
c_{i_0}|p_0(0)| = c_{i_0}|f_{i_0}(0)p_0(0)| \leq \bigg\Vert f_{i_0}p_0 + \sum_{j=1}^M f_{i_j}p_j \bigg\Vert_{\BB}.
\end{gather}

\noindent To that end, fix $ 0 \leq M < \infty $, and consider $1 \leq i_0 < \infty$ and $1 \leq i_1 < i_2 < \cdots < i_M $ satisfying $i_j \neq i_0$ for $j \neq 0$ as well as polynomials $p_0,p_1,\ldots,p_M$. To simplify notation we denote $U_{n,i} := U_{s(n,i)}$ and $r_{n,i} := r_{s(n,i)}$. For $n \geq i_0$ we have that:

\begin{align}
\bigg\Vert f_{i_0}p_0 + \sum_{j=1}^Mf_{i_j}p_j \bigg\Vert_{\BB} &\geq \sup_{|z| = r_{n,i_0}}(1-|z|^2)|(f_{i_0}p_0 + \sum_{j=1}^Mf_{i_j}p_j)'| \geq \notag\\
&\sup_{|z|=r_{n,i_0}}(1-|z|^2)|p_0(z)s(n,i_0)z^{s(n,i_0)-1}|  \\
&- \sup_{|z|=r_{n,i_0}}(1-|z|^2)|p_0(z) \sum_{\stackrel{k \geq i_0}{k \neq n}}s(k,i_0)z^{s(k,i_0)-1}| \\
&- \sup_{|z|=r_{n,i_0}}(1-|z|^2)|p'_0(z)f_{i_0}(z)| \\
&- \sum_{j=1}^M\sup_{|z|=r_{n,i_0}}(1-|z|^2)|p_j(z)f'_{i_j}(z)| \\
&- \sum_{j=1}^M\sup_{|z|=r_{n,i_0}}(1-|z|^2)|p'_j(z)f_{i_j}(z)|.
\end{align}

\noindent For (5) we have that:

\begin{equation}
\sup_{|z|=r_{n,i_0}}(1-|z|^2)|p_0(z)s(n,i_0)z^{s(n,i_0)-1}| = U_{n,i_0}(r_{n,i_0}) \cdot \sup_{|z| = r_{n,i_0}}|p_0(z)|.
\end{equation}

\noindent Note that $U_{n,i_0}(r_{n,i_0}) \ra \frac{1}{\sqrt{e}}$, as $n \ra \infty$. The maximum principle guarantees  that \\ $\sup_{|z| = r_{n,i_0}}|p_0(z)| \ra \no{p_0}_{\infty}$, as $n \ra \infty$. For (6) we have that:

\begin{equation}
\sup_{|z|=r_{n,i_0}}(1-|z|^2)|p_0(z) \sum_{k \neq n}s(k,i_0)z^{s(k,i_0)-1}| \leq \no{p_0}_{\infty} \cdot \sum_{k \neq n} U_{k,i_0}(r_{n,i_0}) \leq \no{p_0}_{\infty} \cdot \sum_{k \neq n}\frac{1}{2^{k+n}},
\end{equation}

\noindent where we used property \textit{(iii)} of Lemma 2.1. The quantities in (7) and (9) can be treated alike. If $0 \leq j \leq M$, then:

\begin{equation}
\sup_{|z|=r_{n,i_0}}(1-|z|^2)|p'_j(z)f_{i_j}(z)| \leq \no{p'_j}_{\infty} \cdot \sup_{|z|=r_{n,i_0}}(1-|z|^2)|f_{i_j}(z)|.
\end{equation}

\noindent Since functions in the Bloch space grow at most logarithmically, we obtain that \\ $\sup_{|z|=r_{n,i_0}}(1-|z|^2)|f_{i_j}(z)| \stackrel{n \ra \infty}{\longrightarrow} 0$. 

\newpage

Finally, for (8) we have that:

\begin{gather}
\sum_{j=1}^M\sup_{|z|=r_{n,i_0}}(1-|z|^2)|p_j(z)f'_{i_j}(z)| \leq \sum_{j=1}^M \no{p_j}_{\infty} \cdot \sup_{|z|=r_{n,i_0}}(1-|z|^2)|f'_{i_j}(z)|  \notag \\
\leq \sum_{j=1}^M \no{p_j}_{\infty} \cdot \sum_{k = i_j}^{\infty}U_{k,i_j}(r_{n,i_0}) \leq \sum_{j=1}^M \no{p_j}_{\infty} \cdot \sum_{k = i_j}^{\infty} \frac{1}{2^{k+n}}.
\end{gather}

\noindent where we used Lemma 2.1 once again. By substituting (10)-(13) into (5)-(9), we get that:

\begin{gather*}
U_{n,i_0}(r_{n,i_0})  \cdot \sup_{|z| = r_{n,i_0}}|p_0(z)| \leq \bigg\Vert f_{i_0}p_0 + \sum_{j=1}^Mf_{i_j}p_j \bigg\Vert_{\BB} + \no{p_0}_{\infty} \cdot \sum_{k \neq n}\frac{1}{2^{k+n}} \\ + \sum_{j=0}^M \no{p'_j}_{\infty} \cdot \sup_{|z|=r_{n,i_0}}(1-|z|^2)|f_{i_j}(z)| + \sum_{j=1}^M \no{p_j}_{\infty} \cdot \sum_{k = i_j}^{\infty} \frac{1}{2^{k+n}} \\\ \leq \bigg\Vert f_{i_0}p_0 + \sum_{j=1}^Mf_{i_j}p_j \bigg\Vert_{\BB} + \no{p_0}_{\infty} \cdot \frac{1}{2^n} + \sum_{j=0}^M \no{p'_j}_{\infty} \cdot \sup_{|z|=r_{n,i_0}}(1-|z|^2)|f_{i_j}(z)| + \frac{1}{2^n} \cdot \sum_{j=1}^M \no{p_j}_{\infty}.
\end{gather*}

\noindent By letting $n \ra \infty$ we get that:

\begin{gather*}
\frac{1}{\sqrt{e}} \cdot \no{p_0}_{\infty} \leq \bigg\Vert f_{i_0}p_0 + \sum_{j=1}^Mf_{i_j}p_j \bigg\Vert_{\BB} ,
\end{gather*}

\noindent and hence:

\begin{gather}
\frac{1}{\sqrt{e}} \cdot |p_0(0)| \leq \frac{1}{\sqrt{e}} \cdot \no{p_0}_{\infty} \leq \bigg\Vert f_{i_0}p_0 + \sum_{j=1}^Mf_{i_j}p_j \bigg\Vert_{\BB}.
\end{gather}

\noindent And we obtained inequality (4) with $c_{i_0} = \frac{1}{\sqrt{e}}. \quad\quad\quad\quad\quad\quad\quad\quad\quad\quad\quad\quad\quad\quad\quad\quad\quad\quad\quad\quad\quad\quad\, \blacksquare$

%%%%%%%%%%%%%%%%%%%%%%%%%%%%%%%%%%%%%%%%
%
%
% END OF PROOF OF THEOREM 1.3
%
%
%%%%%%%%%%%%%%%%%%%%%%%%%%%%%%%%%%%%%%%%

\vspace{5mm}

The above proof may serve as a model to prove Theorem 1.4. The challenge that arises is to define a continuum of functions with the properties described above, instead of countably many.

%%%%%%%%%%%%%%%%%%%%%%%%%%%%%%%%%%%%%%%%%%%
%
%
% PROOF OF THEOREM 1.4
%
%
%%%%%%%%%%%%%%%%%%%%%%%%%%%%%%%%%%%%%%%%%%%

\textit{Proof of Theorem 1.4.}

\vspace{2mm}

We apply Lemma 2.1 and set $s_n = s(n,1)$, $n \geq 1$. The properties of the sequence may be summarized as follows:

\begin{enumerate}[label = (\roman*')]
\item $1 < s_1 < s_2 < \cdots $.

\vspace{1mm}

\item $\frac{s_{n+1}}{s_n} \geq 2 $, for all $n \in \N$.

\vspace{1mm}

\item For every $n \neq n'$ we have $U_{s_n}(r_{s_{n'}}) < \frac{1}{2^{n+n'}}$.
\end{enumerate}

\noindent To define an invariant subspace, we are in need of a lemma that we will use without proof, as a detailed proof can be found in Lemma 3.2, \cite{Niwa2003}.

\begin{lemma}
There exists a family $\{N_{\alpha} \sbs \N \, : \, \alpha \in [0,1] \}$ such that for every $M \in \N$ and every finitely many indices $\alpha_0,\alpha_1, \ldots, \alpha_M \in [0,1]$ with $\alpha_0 \neq \alpha_i$ for $1 \leq i  \leq M$ we have:

\begin{gather}
\text{card}\bigg( N_{\alpha_0} \backslash \bigcup_{i=1}^M N_{\alpha_i} \bigg) = \infty.
\end{gather}

\end{lemma}

\noindent We consider the family $\{N_{\alpha}\}_{\alpha \in [0,1]}$ and we define for each $\alpha \in [0,1]$ a function:

\begin{gather}
f_{\alpha}(z) = 1 + \sum_{n\in N_{\alpha}}z^{s_n} \,\,, z\in \D.
\end{gather}

\noindent Once again, these functions are lacunary, and they belong to the Bloch space. The invariant subspace is defined similarly, $E = \bigvee_{\alpha \in [0,1]}[f_{\alpha}]$ and we aim to show that $\text{ind}(E) = \dim(E/zE) = \text{card}([0,1])$. By definition, it suffices to show that for any $\alpha_0 \in [0,1]$ we have:

\begin{gather*}
f_{\alpha_0} + zE \notin \overline{\text{span}} \bigg\{ f_{\alpha} + zE \, :\, \alpha \in [0,1] \, , \alpha \neq \alpha_0 \bigg\},
\end{gather*}

\noindent where the closure is in the quotient topology of $E/zE$, and $f_{\alpha}+ zE$ denotes the equivalence class of $f_{\alpha}$ in the quotient space. Let $\alpha_0 \in [0,1]$. It is sufficient to find some constant $c >0$ such that:

\begin{gather}
\no{f_{\alpha_0} + zE + u }_{E/zE} \geq c \,\,,\\
\text{for every}\,\, u \in \text{span} \bigg\{ f_{\alpha} + zE \, : \, \alpha \in [0,1] \, , \alpha \neq \alpha_0 \bigg\} \notag
\end{gather}

\noindent To that end, consider any finite number of indices $\alpha_1, \ldots, \alpha_M \in [0,1]$ with $\alpha_i \neq \alpha_0$, and any polynomials $q,p_1,\ldots,p_M$. Then we need to show that 

\begin{gather*}
\no{f_{ \alpha_0}(1 + z \cdot q) + \sum_{i=1}^M p_i f_{ \alpha_i}}_{\BB} \geq c_0.
\end{gather*}

\noindent Set $p_0 = 1 + zq$. From Lemma 2.2 we have that there exists some increasing sequence $k_n \sbs N_{\alpha_0}$ such that $k_n \notin \bigcup_{i=1}^M N_{\alpha_i}$. Starting once again as in Theorem 1.3 we obtain:

\begin{align*}
\bigg\Vert f_{\alpha_0}p_0 + \sum_{j=1}^M f_{\alpha_j}p_j \bigg\Vert_{\BB} &\geq \sup_{|z| = r_{k_n}}(1-|z|^2)|(f_{\alpha_0}p_0 + \sum_{j=1}^M f_{\alpha_j}p_j)'| \geq \notag\\
&\sup_{|z|=r_{k_n}}(1-|z|^2)|p_0(z)s_{k_n}z^{s_{k_n}-1}|  \\
&- \sup_{|z|=r_{k_n}}(1-|z|^2)|p_0(z) \sum_{m \neq k_n}s_mz^{s_m-1}| \\
&- \sup_{|z|=r_{k_n}}(1-|z|^2)|p'_0(z)f_{\alpha_0}(z)| \\
&- \sum_{j=1}^M\sup_{|z|=r_{k_n}}(1-|z|^2)|p_j(z)f'_{\alpha_j}(z)| \\
&- \sum_{j=1}^M\sup_{|z|=r_{k_n}}(1-|z|^2)|p'_j(z)f_{\alpha_j}(z)|.
\end{align*}

\noindent Since $k_n \neq m $ for all indices $m$ appearing in the sums above, condition \textit{(iii')} is satisfied, and hence we may replicate the argument of Theorem 1.3. Therefore by letting $k_n \ra \infty$ we obtain:

\begin{gather*}
\frac{1}{\sqrt{e}} \cdot \no{p_0}_{\infty} \leq \bigg\Vert f_{\alpha_0}p_0 + \sum_{j=1}^M f_{\alpha_j}p_j \bigg\Vert_{\BB},
\end{gather*}

\noindent and since $\no{p_0}_{\infty} = \no{1+zq}_{\infty} \geq 1$ we  arrive at:

\begin{gather*}
\frac{1}{\sqrt{e}} \leq \bigg\Vert f_{\alpha_0}p_0 + \sum_{j=1}^M f_{\alpha_j}p_j \bigg\Vert_{\BB},
\end{gather*}

\noindent which is inequality (17) with $c=\frac{1}{\sqrt{e}}. \quad\quad \quad\quad\quad\quad\quad\quad\quad\quad\quad\quad\quad\quad\quad\quad\quad\quad\quad\quad\quad\quad\quad\quad\quad\quad \blacksquare$

%%%%%%%%%%%%%%%%%%%%%%%%%%%%%%%%%%%%%%%%
%
%
% END OF PROOF OF THEOREM 1.4
%
%
%%%%%%%%%%%%%%%%%%%%%%%%%%%%%%%%%%%%%%%%

\vspace{5mm}

We finish this section by proving Lemma 2.1.

%%%%%%%%%%%%%%%%%%%%%%%%%%%%%%%%%%%%%%%%%%%
%
%
% PROOF OF LEMMA 2.1
%
%
%%%%%%%%%%%%%%%%%%%%%%%%%%%%%%%%%%%%%%%%%%%

\noindent \textit{Proof of Lemma 2.1.}

\vspace{2mm}

We will construct the sequence $s(n,i)$ inductively. We may define $s(1,1) > 1$ as we like. To define the integer $s(2,1)$, we take it to be $s(2,1) \geq 2 s(1,1)$ so that it satisfies conditions \textit{(i)} and  \textit{(ii)}. For condition \textit{(iii)}, we notice that 

\begin{align}
\text{For every} \,\, n\in \N, \,\, U_n(r) \ra 0, \,\, \text{as} \,\,r \ra 1^- \,\, \text{and} \\
\text{for every} \,\, r \in (0,1), \,\, U_n(r) \ra 0, \,\, \text{as} \,\, n \ra \infty.
\end{align}

\noindent This means we can choose $s = s(2,1)$ large enough so that:

\begin{align*}
U_{2,1}(r_{1,1}) < \frac{1}{2^{2+1}} \,\, \text{and} \,\,
U_{1,1}(r_{2,1}) < \frac{1}{2^{2+1}}.
\end{align*}

\noindent To continue, assume that we have already defined all terms up to $s(n,i)$ for some $(n,i)$. The next term has the form $s(n+1,1)$, if $n=i$, and has the form $s(n,i+1)$ if $n>i$. Without loss of generality, we may assume the first case. First choose $s(n+1,1)$ large enough so that $s(n+1,1) > s(n,n)$ and $s(n+1,1) \geq 2 s(n,1)$. That way conditions \textit{(i)} and \textit{(ii)} are taken care of. For the last one we take $s(n+1,1)$ additionally as large as to have:

\begin{align*}
U_{n+1,1}(r_{m,i}) < \frac{1}{2^{n+1+m}} \,\, \text{for all} \,\, 1 \leq m \leq n \,\, \text{and} \,\, 1 \leq i \leq m \\
U_{m,i}(r_{n+1,1}) < \frac{1}{2^{n+1+m}} \,\, \text{for all} \,\, 1 \leq m \leq n \,\, \text{and}\,\, 1 \leq i \leq m.
\end{align*}

\noindent This is possible because of (25), (26) and because we have only finitely many predefined terms for which we need to verify the inequalities. By the inductive hypothesis we can construct the whole sequence $s(n,i)$ satisfying all properties \textit{(i)},\textit{(ii)}
and \textit{(iii)}. The proof is complete.$\quad\quad\quad\quad\quad\quad\quad\quad\quad \blacksquare$

%%%%%%%%%%%%%%%%%%%%%%%%%%%%%%%%%%%%%%%%
%
%
% END OF PROOF OF LEMMA 2.1
%
%
%%%%%%%%%%%%%%%%%%%%%%%%%%%%%%%%%%%%%%%%

\begin{remark}
Lemma 2.1 may be adapted for the spaces $\BB_{\alpha}$, $0<\alpha < 1$, which are Bloch type spaces with norm $|f(0)|+\sup_{|z|<1}(1-|z|^2)^{\alpha}|f'(z)|$. One can then prove Theorem 1.3 and 1.4 for those spaces.
\end{remark}

\section{Invariant subspaces in the little Bloch space}

In this section we will construct invariant subspaces in the little Bloch space of arbitrary, but countable, index. The section is split into two parts. We begin by stating some preliminary results, that will be used in the construction. Then we pass to the construction of the functions generating the invariant subspaces and prove that it has the correct index. The proof of the desired index is achieved with the aid of Lemma 1.2. We denote by $m$ the normalized Lebesgue measure of $[0,2\pi]$ or $\T = \{z \in \C \, : \, |z| = 1\}$.

\subsection{Preliminary results}

\noindent The first result is known as Makarov's inequality and it will be useful to us since it permits us to pass from estimates involving Bloch norms, to estimates involving $L^p$-means . Following that, is an exponential version of it, which follows by a simple calculation.

\begin{theorem}{(Makarov's Inequality)}
Let $g \in \BB$. Then for every $0 \leq r < 1$ and every $n \in \N$ the following inequality holds:

\begin{equation*}
\bigg( \intT{|g(r\zeta)|^{2n}} \bigg)^{\frac{1}{2n}} \leq \no{g}_{\BB} \Bigl( 1 + (n!)^{\frac{1}{2n}} \sqrt{\log \frac{1}{1-r}}\Bigr).
\end{equation*}

\end{theorem}

\noindent A proof of the above can be found in Makarov's original paper \cite{Makarov1985}. Theorem 8.9 in \cite{Pommerenke1992} provides a refined version involving the above numerical constants.

\begin{proposition}{(Makarov's Inequality; exponential form)}
Let $g \in \BB$, with $\no{g}_{\BB} \leq 1 $. Then for every $1-\frac{1}{e}\leq r<1$ we have:

\begin{gather*}
\int_{r\T}\exp\bigg\{\frac{|g(\zeta)|^2}{8\log\frac{1}{1-r}} \bigg\} \,dm(\zeta) \leq 2
\end{gather*}
\end{proposition}

\begin{proof}
From Makarov's inequality for the function $g$ we obtain for every $n \in \N$:

\begin{gather*}
\int_{r\T}|g|^{2n}\,dm \leq \no{g}_{\BB}^{2n}\biggl( 1 + (n!)^{\frac{1}{2n}} \sqrt{\log\frac{1}{1-r}}\biggr)^{2n} 
\end{gather*}

When $r \geq 1 - \frac{1}{e}$ we get $1 + (n!)^{\frac{1}{2n}} \sqrt{\log\frac{1}{1-r}} \leq 2(n!)^{\frac{1}{2n}} \sqrt{\log\frac{1}{1-r}}$, and so:

\begin{gather*}
\int_{r\T}|g|^{2n}\,dm \leq 4^n n! \log^n \frac{1}{1-r}
\end{gather*}

Hence

\begin{gather*}
\int_{r\T} \bigg( \frac{|g|^2}{8\log\frac{1}{1-r}} \bigg)^n \frac{1}{n!} \,dm \leq \frac{1}{2^n}
\end{gather*}

Using the monotone convergence theorem and summing up all the integrals, we get the result.

\end{proof}

Theorem 3.3 is a classic result of Salem and Zygmund, as in their original paper \cite{Salem1947}. 

\begin{theorem}{(Salem-Zygmund)}
Let $F(z) = \sum_{k=1}^{\infty}c_kz^{n_k}$, $z \in \T$, be a lacunary power series, i.e. $\frac{n_{k+1}}{n_k} \geq q > 1 $ for some $q >1$ and all $k \in \N$. Let $C_N = \frac{1}{\sqrt{2}}(|c_1|^2 + \cdots + |c_N|^2)^{\frac{1}{2}}$ and let $R_N,I_N$ be the real and imaginary parts of the $N$-th partial sum respectively. If $C_N \ra +\infty$ and $\frac{c_N}{C_N} \ra 0$, then:

\begin{equation*}
m\bigg( \bigg\{ z\in \T \, : \, \frac{R_N(z)}{C_N} \leq x , \frac{I_N(z)}{C_N} \leq y \bigg\}\bigg) \ra \frac{1}{2\pi} \int_{-\infty}^x \int_{-\infty}^y e^{-\frac{1}{2}(t^2+s^2)}\,dt\,ds, \quad x,y \in \R.
\end{equation*}

\end{theorem}

\noindent We will need to work with the modulus of a complex function instead of the real and imaginary part. We thus compute the asymptotic distribution of it.

\begin{proposition}
Let $F_N$ be the partial sum of the series $F(z) = \sum_{k=1}^{\infty}c_kz^{n_k}$, $z \in \T$. Under the hypotheses of Theorem 3.3 we have that:

\begin{equation*}
m\bigg( \bigg\{ z \in \T \, : \, \frac{|F_N(z)|}{C_N} \leq x \bigg\}\bigg) \ra 1 - e^{-\frac{x^2}{2}}, \quad x \geq 0.
\end{equation*} 

\end{proposition}

%%%%%%%%%%%%%%%%%%%%%%%%%%%%%%%%%%%%%%%%%%%
%
%
% PROOF OF PROPOSITION 3.3
%
%
%%%%%%%%%%%%%%%%%%%%%%%%%%%%%%%%%%%%%%%%%%%

\begin{proof}

Let $U_N = \frac{R_N}{C_N}$ and $V_N = \frac{I_N}{C_N}$. We will compute the distribution of $\sqrt{|U_N|^2+|V_N|^2}$. First of all, for $x \geq 0$ we have $|U_N| \leq x \iff |U_N|^2 \leq x^2$, which means that:

\begin{gather*}
m(|U_N|^2 \leq x^2) \ra \frac{1}{\sqrt{2 \pi}} \int_{-x}^x e^{- \frac{1}{2} t^2} \,dt = \frac{2}{\sqrt{2 \pi}} \int_{0}^x e^{- \frac{1}{2} t^2} \,dt
\end{gather*}

\noindent And thus by a change of variables:

\begin{equation*}
m(|U_N|^2 \leq x) \ra \frac{1}{\sqrt{2 \pi}} \int_{0}^x \frac{1}{\sqrt{t}} e^{- \frac{t}{2}} \,dt =
\frac{1}{\sqrt{2 \pi}} \int_{-\infty}^x \frac{1}{\sqrt{t}} e^{- \frac{t}{2}} \ind_{ \{t \geq 0\} } \ \,dt
\end{equation*}

\noindent The same can be said about the distribution of $|V_N|^2$. Since the joint distribution of $U_N$ and $V_N$ is asymptotically Gaussian, the variables are asymptotically independent. That means that we can use the continuous mapping theorem, to compute the asymptotic distribution of $|U_N|^2 + |V_N|^2$ by convoluting the density functions of each of  the summands. Therefore the density of $|U_N|^2 + |V_N|^2$ will be given by the formula:

\begin{align*}
\varrho(x) &=  \int_{-\infty}^{\infty} \frac{1}{2\pi} \frac{1}{\sqrt{x-y}}\frac{1}{\sqrt{y}}\ind_{ \{y \geq 0 \} } \ind_{ \{x \geq y \} } e^{-\frac{(x-y)}{2}} e^{-\frac{y}{2}} \,dy \\
&= \frac{e^{-\frac{x}{2}}}{2\pi} \int_{0}^x \frac{1}{\sqrt{x-y}}\frac{1}{\sqrt{y}} \,dy , \quad x \geq 0.
\end{align*}

\noindent We substitute $\sqrt{y} = u $ to obtain:

\begin{gather*}
\varrho(x) = \frac{e^{-\frac{x}{2}}}{2\pi} \int_0^{\sqrt{x}} \frac{2}{\sqrt{x}} \frac{1}{ \sqrt{1 - (\frac{u}{\sqrt{x}})^2}} \,du, \quad x \geq 0.
\end{gather*}

\noindent Then we substitute $\frac{u}{\sqrt{x}} = \sin t	$ to get:

\begin{gather*}
\varrho(x) = \frac{e^{-\frac{x}{2}}}{2\pi} \int_0^{\frac{\pi}{2}}
\frac{2}{\sqrt{x}}\frac{ \sqrt{x} \cos t}{\cos t } \,dt = \frac{e^{-\frac{x}{2}}}{2} \, ,  \quad x \geq 0.
\end{gather*}

\noindent What we computed means that:

\begin{gather*}
m( |U_N|^2 + |V_N|^2 \leq x	) \ra \int_0^x \frac{1}{2} e^{-\frac{t}{2}}\,dt , \quad x \geq 0.
\end{gather*}

\noindent Hence a final change of variables gives us that:

\begin{gather*}
m \bigg( \frac{|F_N|}{C_N} \leq x \bigg) \ra \int_0^x t e^{-\frac{1}{2}t^2}\,dt = 1-e^{-\frac{x^2}{2}} , \quad x \geq 0.
\end{gather*}

\end{proof}

%%%%%%%%%%%%%%%%%%%%%%%%%%%%%%%%%%%%%%%%
%
%
% END OF PROOF OF PROPOSITION 3.3
%
%
%%%%%%%%%%%%%%%%%%%%%%%%%%%%%%%%%%%%%%%%

 We introduce a family of lacunary polynomials which will be the basic building blocks to construct functions that generate invariant subspaces of high index. Let $f_s(z) = \sum_{m=s}^{2s} z^{3^m}$, $s \in \N$ be a lacunary polynomial. We associate to every $f_s$ a radius $r_s = 1 - \frac{1}{3^{2s}}$, and a function $X_s(r) = \frac{1}{\sqrt{\log \frac{1}{1-r}}} f_s(r)$, $0<r<1$. 

\begin{lemma}
Let $f_s$ and $r_s$ be as above. There exists a constant $C>0$ such that:

\begin{gather*}
\frac{1}{C} \cdot s \leq \Vert f_s(r_s \cdot) \Vert_{L^2(\T)}^2 \leq C \cdot s \,\, , s \geq 1.
\end{gather*}

\end{lemma}

%%%%%%%%%%%%%%%%%%%%%%%%%%%%%%%%%%%%%%%%
%
%
%PROOF OF Lemma 3.5
%
%
%%%%%%%%%%%%%%%%%%%%%%%%%%%%%%%%%%%%%%%%

\begin{proof}

By Parseval's formula we have:

\begin{gather*}
\Vert f_s(r_s \cdot) \Vert_{L^2}^2 = \sum_{m=s}^{2s}r_s^{2 \cdot 3^m} \leq s+1
\end{gather*}

On the other hand,

\begin{gather*}
\sum_{m=s}^{2s}r_s^{2 \cdot 3^m} \geq \sum_{m=s}^{2s}r_s^{2 \cdot 3^{2s}} = \sum_{m=s}^{2s} \bigg(1 - \frac{1}{3^{2s}} \bigg)^{2 \cdot 3^{2s}} = \bigg(1 - \frac{1}{3^{2s}} \bigg)^{2 \cdot 3^{2s}} (2s - s +1).
\end{gather*}

Since $(1 - \frac{1}{n})^{2n}$ converges to $\frac{1}{e^2}$ as $n \ra + \infty$ we get the reverse inequality.

\end{proof}

%%%%%%%%%%%%%%%%%%%%%%%%%%%%%%%%%%%%%%%%
%
%
% END OF PROOF OF LEMMA 3.5
%
%
%%%%%%%%%%%%%%%%%%%%%%%%%%%%%%%%%%%%%%%%

\begin{lemma}
Let $f_s$ and $r_s$ be as above. There exists a constant $c > 0$ such that for every $\eps >0$ and every $M > 0$ there exists arbitrary large $s \in \N$ with the property that for every $0 \leq x \leq M$ we have

\begin{gather*}
m \bigg( \bigg\{ \zeta \in \T \, : \, |f_s(r_s \zeta)| > x \sqrt{\log \frac{1}{1-r_s}}  \bigg\}\bigg) \geq e^{-cx^2} - \eps.
\end{gather*}

\end{lemma}

%%%%%%%%%%%%%%%%%%%%%%%%%%%%%%%%%%%%%%%%
%
%
%PROOF OF Lemma 3.6
%
%
%%%%%%%%%%%%%%%%%%%%%%%%%%%%%%%%%%%%%%%%

\begin{proof}
We wish to apply Theorem 3.4. First we notice that:

\begin{gather*}
\sqrt{\log \frac{1}{1-r_s}} = \sqrt{2s \log3} \asymp \Vert f_s(r_s \cdot)\Vert_{L^2} \ra + \infty.
\end{gather*}

Following the notation of Theorem 3.4, the above means that:

\begin{gather*}
C_{2s} = \frac{1}{\sqrt{2}}(|c_1|^2 + |c_2|^2 + \cdots + |c_{2s}|^2)^{\frac{1}{2}} = \frac{1}{\sqrt{2}} \Vert f_s(r_s \cdot)\Vert_{L^2} \ra + \infty
\end{gather*}

Moreover it is clear that all non-zero Taylor coefficients are bounded. This means that both conditions of Theorem 3.4 are satisfied. The fact that we can apply Theorem 3.4 to a lacunary polynomial of this form comes from the fact that the convergence in Salem and Zygmund's Theorem is guaranteed only by the ``length'' $s+1$ of the block, and the fact that the lacunary gap on the exponents is at least 3. Finally, the convergence is uniform in $x$ whenever it belongs to a fixed bounded interval. 
\end{proof}

%%%%%%%%%%%%%%%%%%%%%%%%%%%%%%%%%%%%%%%%
%
%
% END OF PROOF OF LEMMA 3.6
%
%
%%%%%%%%%%%%%%%%%%%%%%%%%%%%%%%%%%%%%%%%

We finish this section by observing the following property of the functions $X_s(r)$: 

\begin{align}
\text{For every} \,\, s \in \N, \,\, X_s(r) \ra 0 , \,\, \text{as} \,\, r \ra 1^- \,\, \text{and} \\
\text{for every}\,\, r \in (0,1), \,\, X_s(r) \ra 0, \,\, \text{as} \,\, s \ra \infty.
\end{align}

These properties are analogous to (18) and (19). Property (20) is straightforward. For property (21), fix $r \in (0,1)$. Then:

\begin{gather*}
X_s(r) = \frac{1}{\sqrt{\log \frac{1}{1-r}}} \sum_{m=s}^{2s}r^{3^m} \leq \frac{1}{\sqrt{\log \frac{1}{1-r}}} \sum_{m=s}^{2s}r^{3^s} = \frac{1}{\sqrt{\log \frac{1}{1-r}}} (s+1) r^{3^s}
\end{gather*}

And $(s+1)r^{3^s}$ converges to zero as $s \ra \infty$.

\subsection{Invariant subspaces in $\BB_0$ of arbitrary index - Construction}

We consider the following functions:

\vspace{2mm}

\begin{gather*}
f_i(z) = 1 + \sum_{j=i}^{\infty} \delta_j f_{i,j}(z) \,\, , z \in \D, \,\, 1 \leq i < + \infty ,
\end{gather*}

\noindent where $f_{i,j}$ are the blocks

\begin{gather*}
f_{i,j}(z) = \sum_{m = s(i,j)}^{2s(i,j)} z^{3^m} \,\, , z\in \D,
\end{gather*}

\noindent and $\delta_j = 2^9 \cdot \sqrt{\frac{c}{j}}$, with $c$ the constant appearing in Lemma 3.6. Since every function is lacunary and the coefficients $\delta_j$ tend to zero, these functions belong to $\BB_0$. To each block are assigned its associated parameters:

\begin{itemize}

\item $\fa \, i \geq 1 \, , \fa \, j \geq i  $ set
\begin{gather*}
r(i,j) = 1 - \frac{1}{3^{2s(i,j)}}
\end{gather*}

\vspace{1mm}

\item $\fa  \, i \geq 1 \, , \fa j \geq i \, ,  \fa \, r \in (0,1)  $ set
\begin{gather*}
X_{i,j}(r) = \frac{1}{\sqrt{\log \frac{1}{1-r}}} \sum_{m=s(i,j)}^{2s(i,j)}r^{3^m}
\end{gather*}

\end{itemize}

\begin{proposition}
There exists a sequence $\{s(i,j)\}_{1 \leq i \leq j < \infty} \sbs \N$, such that:

\begin{itemize}

\item[$(1.i.j)$] For all $i \in \N$ and $j \geq i$ we have $s(i,j) > \max\{ 2^{4j+4} , s(i',j') \, : \, i' < i , \, j' < j  \}$

\vspace{1mm}

\item[$(2.i.j)$] For all $i \in \N$ and $j \geq i$ and $(i',j') \neq (i,j)$ we have 
\begin{gather*}
X_{i,j}(r(i',j')) \leq \frac{1}{\delta_j} \cdot \frac{1}{2^{i+i'+j+2j'+2}}
\end{gather*}

\vspace{1mm}

\item[$(3.i.j)$] For all $i \in \N$ and $j \geq i$ we have 

\begin{gather*}
m\bigg( \bigg\{ \zeta \in\T : |f_{i,j}(r(i,j) \zeta)| > \sqrt{\frac{1}{c \cdot 2^{j+6}}} \cdot \sqrt{\log \frac{1}{1-r(i,j)}}   \bigg\} \bigg) \geq  1 -  \frac{1}{2^{j+5}},
\end{gather*}

where $c$ is the constant appearing in Lemma 3.6.

\end{itemize}
\end{proposition}

Set $E_{i,j} = \bigg\{\zeta \in \T \, :\, |f_{i,j}(r(i,j)\zeta)| \geq \sqrt{\frac{1}{c \cdot 2^{j+6}}} \cdot \sqrt{  \log\frac{1}{1-r(i,j)}} \bigg\}$ and $F_{i,j} = \T \backslash E_{i,j}$

\begin{proof}
We define the sequence inductively, respecting the lexicographic order. It is clear that to obtain the above conditions it just suffices to make sure that $s(i,j)$ is large enough at each step. In particular, for condition \textit{(3.i.j)} we apply Lemma 3.6 with $x = \sqrt{\frac{1}{c \cdot 2^{j+6}}}$, $\eps = \frac{1}{2^{j+6}}$ and apply the inequality $e^{-x} > 1 - x $ which holds for $x<1$.  

\end{proof}

%%%%%%%%%%%%%%%%%%%%%%%%%%%%%%%%%%%%%%%%%%%
%
%
% PROOF OF THEOREM 1.5
%
%
%%%%%%%%%%%%%%%%%%%%%%%%%%%%%%%%%%%%%%%%%%%

\noindent \textit{Proof of Theorem 1.5.}

\vspace{2mm}

As in the proof of Theorem 1.3, we define again $E_N = \bigvee_{1 \leq i \leq N}[f_i] \sbs \BB_0$ and $E_{\infty} =\bigvee_{1 \leq i < \infty}[f_i] \sbs \BB_0$ where $f_i$, $1 \leq i < \infty$ are the functions described in the beginning of section 3.2 and the implicit sequence $\{ s(i,j)\}_{1 \leq i\leq j < \infty}$ satisfies the conditions of Proposition 3.7. Following a similar argument as in Theorem 1.3, we fix $1 \leq M < \infty$, $1 \leq i_1 < i_2 < \cdots < i_M \in \N $ and $p_1,p_2, \ldots, p_M $ be polynomials, such that $\no{\sum_{m=1}^M p_m f_{i_m}}_{\BB} \leq 1$. We wish to apply Lemma 2.1 to the function $\sum_{m=1}^M p_m f_{i_m}$ in order to bound the values $p_m(0)$, $1 \leq m \leq M$, by some constants independent of the polynomials. 

\newpage

Define for every $1 \leq m \leq M$ and every $j \geq i_m$ the set:

\begin{gather*}
    U_{i_m,j} = \{ \zeta \in E_{i_m,j} \, : \, |p_m(r(i_m,j)\zeta)| \geq 2^{j-1} \}.
\end{gather*}

We will implement the following scheme. We assume that for some $J > i_M$ we have the following two inequalities:

\begin{gather}
    m(U_{i_m,J}) \leq \frac{J}{2^{J+5}} \,\, , \, 1 \leq m \leq M,  \\
    \no{p_m}_{L^{2^{J+1}}(r(i_m,J+1)\T)} \leq 2^{J+1} \,\, , \, 1 \leq m \leq M,
\end{gather}

and we will prove that :

\begin{gather}
    m(U_{i_m,J-1}) \leq \frac{J-1}{2^{J+4}} \,\, , \, 1 \leq m \leq M,  \\
    \no{p_m}_{L^{2^J}(r(i_m,J)\T)} \leq 2^{J} \,\, , \, 1 \leq m \leq M.
\end{gather}

The fact that we can assume (22) and (23) for all $J$ large enough comes from the fact that the polynomials $p_m$ are bounded on the closed unit disc. Give the success of this argument, we may iterate it, starting from some radius close enough to one until we obtain:

\begin{gather*}
    \no{p_m}_{L^{2^{i_M}}(r(i_m,i_M)\T)} \leq 2^{i_M} \,\, , \, 1 \leq m \leq M,
\end{gather*}

\noindent which means that $|p_m(0)| \leq 2^{i_M}$ for every $1 \leq m \leq M$. The constant does not depend on the polynomials $p_m$, so using Lemma 2.1 we will have proven Theorem 1.5.

\vspace{2mm}
For the rest of the proof, fix some $1 \leq n \leq M$ and some $J > i_M$. First we demonstrate how to obtain inequality (25). Let $A = \{ \zeta \in \T \, : \, |p_n(r(i_n,J)\zeta)| \geq 2^{J-1}\}$. Then: 

\begin{align*}
    m(A) & = m(A \cap E_{i_n,J}) + m(A \cap F_{i_n,J}) \\
         & \leq m(U_{i_n,J}) + m(F_{i_n,J})  \\
         & \leq \frac{J}{2^{J+5}} + \frac{1}{2^{J+5}} \\
         & \leq \frac{J}{2^J}\cdot \frac{1}{2^{5}} + \frac{1}{2^5} \\ 
         & \leq 2 \cdot \frac{1}{2^5} = \frac{1}{16}.
\end{align*}

We may now use Minkowski's inequality to write:

\begin{gather*}
     \no{p_n}_{L^{2^J}(r(i_n,J)\T)}  =  \bigg( \int_{\T \backslash A}|p_n(r(i_n,J)\zeta)|^{2^J} \,dm(\zeta) + \int_{A}|p_n(r(i_n,J)\zeta)|^{2^J} \,dm(\zeta) \bigg)^{\frac{1}{2^J}} \\
     \leq \bigg( \int_{\T \backslash A}|p_n(r(i_n,J)\zeta)|^{2^J} \,dm(\zeta) \bigg)^{\frac{1}{2^J}} + \bigg( \int_{A}|p_n(r(i_n,J)\zeta)|^{2^J} \,dm(\zeta) \bigg)^{\frac{1}{2^J}} .
\end{gather*}

Applying Cauchy-Schwartz inequality to the second integral we obtain:

\begin{gather*}
    \no{p_n}_{L^{2^J}(r(i_n,J)\T)} \leq  \bigg( \int_{\T \backslash A}|p_n(r(i_n,J)\zeta)|^{2^J} \,dm(\zeta) \bigg)^{\frac{1}{2^J}} + \sqrt{m(A)} \cdot  \no{p_n}_{L^{2^{J+1}}(r(i_n,J)\T)}.
\end{gather*}

Using the bound for the polynomial on the set $\T \backslash A$, our hypothesis (23) and the fact that $r(i_n,J) < r(i_n,J+1)$, we get that

\begin{gather*}
    \no{p_n}_{L^{2^J}(r(i_n,J)\T)} \leq  2^{J-1} + \frac{1}{4} \cdot 2^{J+1} = 2^J.
\end{gather*}

Inequality (25) is proven. We proceed to prove inequality (24). By (1.i.j) we see that $s(i,j) > s(1,1) \geq 3 $ and hence all radii $r(i,j)$ satisfy $r(i,j) \geq 1 - \frac{1}{e}$. We may thus apply Proposition 3.2 for the function  $\sum_{m=1}^M p_m f_{i_m}$, on the radius $r(i_n,J-1)$ and get:

\begin{gather*}
    \int_{r(i_n,J-1)\T} \exp \bigg\{ \frac{|\sum_{m=1}^M p_m f_{i_m}|^2}{8\log\frac{1}{1-r(i_n,J-1)}} \bigg\} \,dm \leq 2.
\end{gather*}

Applying Jensen's inequality for the exponential function yields:

\begin{gather*}
   \exp \bigg\{ \int_{r(i_n,J-1)\T}  \frac{|\sum_{m=1}^M p_m f_{i_m}|^2}{8\log\frac{1}{1-r(i_n,J-1)}}\,dm \bigg\}  \leq 2.
\end{gather*}

Using the inequality $|x + y|^2 = |x|^2 + |y|^2 + 2\text{Re}(xy) \geq |x|^2 + |y|^2 - 2|x||y|$ and applying it for $x = p_n\delta_{J-1} f_{i_n,J-1}$ and $y = \sum_{m=1}^M p_m f_{i_m} - p_n\delta_{J-1} f_{i_n,J-1} $ we obtain:

\begin{gather*}
\bigg| \sum_{m=1}^M p_m f_{i_m} \bigg|^2 \geq 
|p_n\delta_{J-1} f_{i_n,J-1}|^2 -  2 |p_n\delta_{J-1} f_{i_n,J-1}| \cdot \bigg| p_n \bigg( \sum_{\stackrel{j \geq i_n}{j \neq J-1}}\delta_j f_{i_n,j} \bigg) + \sum_{\stackrel{1 \leq m \leq M}{m \neq n} } p_m \bigg( \sum_{ j \geq i_m }\delta_j f_{i_m,j}  \bigg) \bigg|.
\end{gather*}

Since we are integrating over the radius $r(i_n,J-1)$ we may use (2.i.j) to get the following bound:

\begin{align*}
    \bigg| p_n \bigg( 1 + \sum_{\stackrel{j \geq i_n}{j \neq J-1}}\delta_j f_{i_n,j} \bigg) + \sum_{\stackrel{1 \leq m \leq M}{m \neq n} } p_m \bigg( 1 + \sum_{ j \geq i_m }\delta_j f_{i_m,j}  \bigg) \bigg| \\
    \leq |p_n| \bigg( 1 +\sum_{\stackrel{j \geq i_n}{j \neq J-1}}\delta_j X_{i_n,j}(r(i_n,J-1)) \cdot \sqrt{\log\frac{1}{1-r(i_n,J-1)}}  \bigg)  \\
    + \sum_{\stackrel{1 \leq m \leq M}{m \neq n} } |p_m| \bigg( 1 + \sum_{ j \geq i_m }\delta_j X_{i_m,j}(r(i_n,J-1)) \cdot \sqrt{\log\frac{1}{1-r(i_n,J-1)}} \bigg) \\ 
    \leq |p_n| \bigg( 1 +\sum_{\stackrel{j \geq i_n}{j \neq J-1}} \frac{1}{2^{2i_n + j + 2(J-1) + 2 }} \cdot \sqrt{\log\frac{1}{1-r(i_n,J-1)}}  \bigg) \\
    + \sum_{\stackrel{1 \leq m \leq M}{m \neq n} } |p_m| \bigg( 1 + \sum_{ j \geq i_m }\frac{1}{2^{i_n + i_m + j + 2(J-1) + 2 }} \cdot \sqrt{\log\frac{1}{1-r(i_n,J-1)}} \bigg).
\end{align*}

A short calculation then yields: 

\begin{align*}
    \bigg| p_n \bigg( 1 + \sum_{\stackrel{j \geq i_n}{j \neq J-1}}\delta_j f_{i_n,j} \bigg) + \sum_{\stackrel{1 \leq m \leq M}{m \neq n} } p_m \bigg( 1 + \sum_{ j \geq i_m }\delta_j f_{i_m,j}  \bigg) \bigg| \leq \sum_{m=1}^M |p_m|\bigg( 1 + \frac{1}{2^{2J+1}}\sqrt{\log\frac{1}{1-r(i_n,J-1)}}  \bigg).
\end{align*}

As a result: 

\begin{gather*}
    \frac{|\sum_{m=1}^M p_m f_{i_m}|^2}{8\log\frac{1}{1-r(i_n,J-1)}} \geq \frac{|p_n\delta_{J-1} f_{i_n,J-1}|^2}{8\log\frac{1}{1-r(i_n,J-1)}} - \frac{1}{4} \cdot \frac{|p_n \delta_{J-1} f_{i_n,J-1}|}{\sqrt{\log\frac{1}{1-r(i_n,J-1)}}} \cdot \bigg( \sum_{m=1}^M |p_m|\bigg(\frac{1}{\sqrt{\log\frac{1}{1-r(i_n,J-1)}}} + \frac{1}{2^{2J+1}} \bigg)  \bigg).
\end{gather*}

Taking into account condition (1.i.j) gives that $\frac{1}{\sqrt{\log\frac{1}{1-r(i_n,J-1)}}} \leq \frac{1}{2^{2J+1}}$. Therefore after integration on the circle of radius $r(i_n,J-1)$ and using Cauchy-Schwartz inequality, we obtain:

\begin{gather*}
    \int_{r(i_n,J-1)\T} \frac{|\sum_{m=1}^M p_m f_{i_m}|^2}{8\log\frac{1}{1-r(i_n,J-1)}} \,dm  \geq  \\
  \int_{r(i_n,J-1)\T} \frac{|p_n\delta_{J-1} f_{i_n,J-1}|^2}{8\log\frac{1}{1-r(i_n,J-1)}} \,dm - \frac{1}{4} \sum_{m=1}^M \frac{1}{2^{2J}} \bigg(\int_{r(i_n,J-1)\T} \frac{|p_n\delta_{J-1} f_{i_n,J-1}|^2}{\log\frac{1}{1-r(i_n,J-1)}} \,dm \bigg)^{\frac{1}{2}}\no{p_m}_{L^2(r(i_n,J-1)\T)} .
\end{gather*}

\newpage

Finally, we notice that because of (1.i.j):

\begin{gather*}
    \no{p_m}_{L^2(r(i_n,J-1)\T)}  \leq \no{p_m}_{L^{2^J}(r(i_n,J-1)\T)}  \leq \no{p_m}_{L^{2^J}(r(i_m,J)\T)} \leq 2^J.
\end{gather*}

Moreover $J > i_M$ so $\frac{M}{2^J} \leq 1$. Therefore we obtain the following inequality:

\begin{gather*}
    2 \geq \exp \bigg\{ \int_{r(i_n,J-1)\T}  \frac{|\sum_{m=1}^M p_m f_{i_m}|^2}{8\log\frac{1}{1-r(i_n,J-1)}}\,dm \bigg\}  \geq \exp(X(X-1)),
\end{gather*}

where $X =\bigg(\int_{r(i_n,J-1)\T} \frac{|p_n\delta_{J-1} f_{i_n,J-1}|^2}{\log\frac{1}{1-r(i_n,J-1)}} \,dm \bigg)^{\frac{1}{2}} $. As a result it must be that $X \leq 2$, which in turn gives:

\begin{gather*}
    \int_{r(i_n,J-1)\T} \frac{|p_n\delta_{J-1} f_{i_n,J-1}|^2}{\log\frac{1}{1-r(i_n,J-1)}} \,dm \leq 32.
\end{gather*}

Restricting ourselves on the set $U_{i_n,J-1}$ and using the growth of the block $f_{i_n,J-1}$ and the size of the polynomial $p_n$ yields exactly the desired result, i.e:

\begin{gather*}
    m(U_{i_n,J-1}) \leq \frac{J-1}{2^{J+4}}.
\end{gather*}

%%%%%%%%%%%%%%%%%%%%%%%%%%%%%%%%%%%%%%%%%%%
%
%
% END OF PROOF OF THEOREM 1.5
%
%
%%%%%%%%%%%%%%%%%%%%%%%%%%%%%%%%%%%%%%%%%%%

\section{Weak star closed invariant subspaces \& Stability of Index}

In this last section we discuss the index of a weak-star ($w^{\ast}$) closed, invariant subspaces of a Banach space. In the first subsection duality in the Bloch spaces is introduced, to provide a concrete example. Following that we extend several of Richter's results from \cite{Richter1987} for the weak-star topology, and prove Theorem 1.6. Finally we may apply that to the Bloch spaces and obtain Theorem 1.7.

\subsection{Duality in the Bloch spaces}

Consider the Bergman space $A^1$ of integrable functions in the unit disc, as well as the following dual pairings:

\begin{gather*}
\langle \cdot , \cdot \rangle : \BB_0 \times A^1 \ra \C \\
\langle f , g \rangle = \lim_{r \ra 1} \int_{\D}f(rz) \overline{g(rz)} \,dA(z),
\end{gather*}

\noindent and,

\begin{gather*}
\langle \cdot , \cdot \rangle : A^1 \times \BB \ra \C \\
\langle g , f \rangle = \lim_{r \ra 1} \int_{\D}g(rz) \overline {f(rz)} \,dA(z).
\end{gather*}

\noindent It is proven in \cite{Hedenmalm2000},\cite{Zhu2014}, \cite{Axler1988} that these pairings are well defined and realize the dualities $( \BB_0 )^{\ast} \cong A^1$ and $(A^1)^{\ast} \cong \BB$. We can therefore endow the space $\BB_0$ with the weak topology inherited from its dual space, and the space $\BB$ with the $w^{\ast}$-topology inherited from its pre-dual. 

\newpage

The topology in $\BB$ can be characterized in terms of nets, as is done in \cite{Anderson1991}. If $\{ f_i \}_{i \in I} \sbs \BB$ is a net, then:

\begin{equation}
f_i \stackrel{w^{\ast}}{\ra} 0 \iff 
\begin{cases}
f_i(z) \ra 0 \,\, , \quad \fa z \in \D \\
\limsup_{i}\no{f_i}_{\BB} < + \infty    
\end{cases}
\end{equation}

\noindent The above statement remains true if we replace $w^{\ast}$-convergence by weak convergence and the net $\{ f_i \}_{i \in I}$ belongs to $\BB_0$. If $f \in \BB_0$ then $f$ is norm-cyclic in $\BB_0$ if and only if $f$ is $w^{\ast}$-cyclic in $\BB$ (\cite{Anderson1991}). Since polynomials are norm-dense in $\BB_0$, the constant function $1$ is norm-cyclic in $\BB_0$ and so $1$ must be $w^{\ast}$-cyclic in $\BB$, or equivalently, polynomials are $w^{\ast}$-dense in $\BB$. In particular, we may consider polynomials belonging in the set $\{ a_0 + a_1z + \cdots + a_N z^N \, : \, N \in \N, a_i \in \Q + i\Q \}$, thus obtaining a countable, norm-dense set in $\BB_0$ and hence a countable, $w^{\ast}$-dense set in $\BB$. This means that the space $(\BB, w^{\ast})$ is separable. 

Our aim is to produce $w^{\ast}$-closed invariant subspaces of arbitrary index in the Bloch space. Since $(\BB, w^{\ast})$ is separable, the index of an invariant subspace can be at most countable. In \cite{Borichev1998} and \cite{Niwa2003} it is proven that $H^{\infty}$ with the norm topology contains invariant subspaces of index equal to the cardinality of the interval $[0,1]$. Theorem 1.4 is the Bloch space equivalent of that. It is also known that if $H^{\infty}$ is equipped with the $w^{\ast}$-topology then Beurling's theorem holds, i.e. for every $E \sbs H^{\infty}$  invariant, then $E = \phi H^{\infty}$ for some inner function $\phi$ (\cite{Gamelin2005}). This implies that all invariant subspaces have the index one property, and thus Theorem 1.7 provides a contrasting phenomenon to the situation in $H^{\infty}$.

We remind the reader of two properties that hold in the Bloch space, and will be used in what follows. If we have a function $f \in \BB$ and $F' = f$ then $F \in \BB_0$ and $\no{F}_{\BB} \leq C \no{f}_{\BB}$ for some constant $C$ independent of $f$. The second one, called the ``\textit{division property}'' states that if $f \in \BB$ and $f(\lambda)=0$ then $\frac{f}{z-\lambda} \in \BB$. In a Banach space of analytic functions (see Lemma 1.2) the division operator $R_{\lambda}$, defined on $E_{\lambda} = \{f \in \BB \, : \, f(\lambda) =0 \}$, is bounded for any $\lambda  \in \D$, which is also equivalent to saying that the operator $M_z - \lambda$ is bounded below for any $\lambda \in \D$ (\cite{Richter1987}) . Next is a useful proposition.

\begin{proposition}
Let $M_z : \BB \ra \BB $ be the Shift operator. Then:
\begin{enumerate}
\item $M_z$ is $w^{\ast} - w^{\ast}$ continuous on $\BB$
\item $R_{\lambda}$ is $w^{\ast} - w^{\ast}$ continuous for every $\lambda \in \D$
\end{enumerate}
\end{proposition}

%%%%%%%%%%%%%%%%%%%%%%%%%%%%%%%%%%%%%%%%%%%
%
%
% PROOF OF PROPOSITION 4.1
%
%
%%%%%%%%%%%%%%%%%%%%%%%%%%%%%%%%%%%%%%%%%%%

\begin{proof}
To prove continuity of $M_z$, consider a converging net $f_i \stackrel{w^{\ast}}{\ra} 0$. We need to show that $zf_i \stackrel{w^{\ast}}{\ra} 0$. Pointwise convergence is obvious. Moreover, 

\begin{gather*}
\no{zf_i}_{\BB} = \no{M_zf_i}_{\BB} \leq \no{M_z}\no{f_i}_{\BB}  \,\, \text{so,} \\
\limsup_{i} \no{zf_i}_{\BB} \leq \no{M_z} \limsup_i \no{f_i}_{\BB} < + \infty
\end{gather*}

The two combined guarantee convergence of the net. For the second claim, consider a converging net $f_i \stackrel{w^{\ast}}{\ra} f$ in $E_{\lambda}$. There are functions $g_i, g \in \BB$ such that $(z - \lambda )g_i = f_i$ and $(z- \lambda)g = f$. This implies that $R_{\lambda}f_i = g_i $ and $R_{\lambda}f = g$, and thus it suffices to show that $g_i\stackrel{w^{\ast}}{\ra}g $. This can be deduced by the boundedness of $R_{\lambda}$ and by proceeding as for the shift operator.

\end{proof} 

%%%%%%%%%%%%%%%%%%%%%%%%%%%%%%%%%%%%%%%%%%%
%
%
% END OF PROOF OF PROPOSITION 4.1
%
%
%%%%%%%%%%%%%%%%%%%%%%%%%%%%%%%%%%%%%%%%%%%

 The above proposition has the following consequence: If $E$ is a $w^{\ast}$-closed subspace of $\BB$, then $M_zE$ is also $w^{\ast}$-closed. This means that it is meaningful to consider the quotient $E / zE$ for an invariant subspace which is $w^{\ast}$-closed, and that will it be a well-defined locally convex space.

\subsection{Extension of Richter's results - Stability of index}

We consider the general situation where $X_0, X, Y$ are Banach spaces of analytic functions satisfying the division property, and that satisfy the following dualities: $X_0^{\ast} \cong Y$ and  $Y^{\ast} \cong X$. We furthermore assume that Proposition 4.1 is true for the space, i.e. $M_z$ and $R_{\lambda}$ are continuous with respect to the $w^{\ast}$-topology in $X$. We will denote by $Lat_{X_0}(M_z)$  the lattice of norm closed, invariant subspaces of $X_0$ and by $Lat_X(M_z,w^{\ast})$ the lattice of $w^{\ast}$-closed, invariant subspaces of $X$. As mentioned above for given $\mathcal{M} \in Lat_X(M_z,w^{\ast})$, quotients of the form $\mathcal{M}/z\mathcal{M}$ make sense under the above assumptions, and the projection operator onto the quotient space is always continuous. The dimension of $\mathcal{M}/z\mathcal{M}$ is the same as that of $\mathcal{M}/(z-{\lambda}\mathcal{M})$ for $\lambda \in \D$, as follows from general properties of the shift operator and Proposition 4.1. If $A$ is any subset of $X$, we denote by $\overline{A}^{w^{\ast}}$ the $w^{\ast}$-closure of $A$ in $X$. This coincides with all the $w^{\ast}$-limits of nets in $A$.  If $f \in X$, we will write $[f]_{\ast}$ for the  $w^{\ast}$-closed, invariant subspace generated by $f$. Moreover, when writing $\mathcal{M} \vee \mathcal{N}$ for $\mathcal{M}, \mathcal{N} \in Lat_X(M_z,w^{\ast})$ we will mean the smallest $w^{\ast}$-closed invariant subspace containing $\mathcal{M} + \mathcal{N}$. Finally, for given $\mathcal{M} \in Lat_X(M_z,w^{\ast})$ we define $\mathcal{Z}(\mathcal{M}) = \{ \lambda \in \D \, | \, f(\lambda) = 0 \, , \, \fa f \in \mathcal{M} \}$. 

\begin{proposition}
Let $\mathcal{M}, \mathcal{N} \in Lat_X(M_z,w^{\ast})$. Then:
\begin{enumerate}
\item $\text{ind}(\mathcal{M} \vee \mathcal{N}) \leq \text{ind}(\mathcal{M}) + \text{ind}(\mathcal{N})$,
\item If  $\text{ind}(\mathcal{M}) = m \geq 2$, with $m$ finite, and $ n_1 +n_2 = m$ then there exist $\mathcal{N}_1, \mathcal{N}_2 \in Lat_X(M_z,w^{\ast})$, $\mathcal{N}_1, \mathcal{N}_2 \sbs \mathcal{M}$, such that $\text{ind}(\mathcal{N}_i) = n_i$ and $\text{ind}(\mathcal{N}_1 \vee \mathcal{N}_2) = m$.
\end{enumerate}
\end{proposition}

%%%%%%%%%%%%%%%%%%%%%%%%%%%%%%%%%%%%%%%%%%%
%
%
% PROOF OF PROPOSITION 4.2
%
%
%%%%%%%%%%%%%%%%%%%%%%%%%%%%%%%%%%%%%%%%%%%

\begin{proof}
For the first implication, if either $\mathcal{M}$ or $\mathcal{N}$ have infinite index, then we have nothing to show, so assume the index of both is finite. In that case there exist $\mathcal{M}_1 \sbs \mathcal{M}$ and $\mathcal{N}_1 \sbs \mathcal{N}$, finite dimensional subspaces, such that:

\begin{gather*}
\mathcal{M} = z\mathcal{M} + \mathcal{M}_1  \,\, , \,\, \mathcal{N} = z\mathcal{N} + \mathcal{N}_1 \,\, \text{and}, \\
\text{ind}(\mathcal{M}) = \dim(\mathcal{M}_1) \,\, , \,\, \text{ind}(\mathcal{N}) = \dim(\mathcal{N}_1).
\end{gather*}

Then 

\begin{gather*}
\mathcal{M} + \mathcal{N} = z(\mathcal{M}+ \mathcal{N}) + (\mathcal{M}_1 + \mathcal{N}_1) \sbs z(\mathcal{M} \vee \mathcal{N}) + (\mathcal{M}_1 + \mathcal{N}_1) \sbs \mathcal{M} \vee \mathcal{N}.
\end{gather*}

Since $\mathcal{M}_1 + \mathcal{N}_1$ is finite dimensional, it is $w^{\ast} $-closed. The space $z(\mathcal{M} \vee \mathcal{N}) + (\mathcal{M}_1 + \mathcal{N}_1)$ is then also $w^{\ast}$-closed. Indeed, consider the natural projection of $X$ onto the quotient $X / z(\mathcal{M} \vee \mathcal{N})$. Consider a base $h_1, \ldots, h_n$ of $\mathcal{M}_1 + \mathcal{N}_1$. That map is well defined because $z(\mathcal{M} \vee \mathcal{N})$ is $w^{\ast}$-closed and continuous. The space spanned by $Ph_1, \ldots, Ph_n$ is finite dimensional, thus closed in the quotient topology which is also locally convex. But looking at the inverse image we see that:

\begin{gather*}
P^{-1}(\text{span}\{Ph_1,\ldots, Ph_n \}) = z(\mathcal{M} \vee \mathcal{N}) + (\mathcal{M}_1 + \mathcal{N}_1).
\end{gather*}

By the continuity of $P$ we deduce that this space is $w^{\ast}$-closed. Since $\mathcal{M} + \mathcal{N}$ is $w^{\ast}$-dense in $\mathcal{M}\vee \mathcal{N}$ we get from the above inclusions that

\begin{gather*}
z(\mathcal{M} \vee \mathcal{N}) + (\mathcal{M}_1 + \mathcal{N}_1) = \mathcal{M} \vee \mathcal{N},
\end{gather*}

and so,
 
\begin{gather*}
\text{ind}(\mathcal{M}\vee \mathcal{N}) = \dim(\mathcal{M}_1 + \mathcal{N}_1) \leq \dim(\mathcal{M}_1) + \dim(\mathcal{N}_1) = \text{ind}(\mathcal{M}) + \text{ind}(\mathcal{N}).
\end{gather*}
 
To prove the second implication we use a similar argument.
 
\end{proof}

%%%%%%%%%%%%%%%%%%%%%%%%%%%%%%%%%%%%%%%%%%%
%
%
% END OF PROOF OF PROPOSITION 4.2
%
%
%%%%%%%%%%%%%%%%%%%%%%%%%%%%%%%%%%%%%%%%%%%

\begin{proposition}
Let $\mathcal{M} \in Lat_X(M_z,w^{\ast})$ and $\lambda \notin \mathcal{Z}(\mathcal{M})$. The following are equivalent:

\begin{enumerate}
\item $\text{ind}(\mathcal{M})=1$,
\item If $f\in \mathcal{M}$ such that $f(\lambda)=0$ then there exists some $h\in \mathcal{M}$ such that $(z-\lambda)h = f$,
\item If $(z-\lambda)h = f \in \mathcal{M}$ for some $h \in X$  then $h \in \mathcal{M}$.
\end{enumerate}

\end{proposition}

%%%%%%%%%%%%%%%%%%%%%%%%%%%%%%%%%%%%%%%%%%%
%
%
% PROOF OF PROPOSITION 4.3
%
%
%%%%%%%%%%%%%%%%%%%%%%%%%%%%%%%%%%%%%%%%%%%

\begin{proof}
The equivalence of \textit{(2)} and \textit{(3)} is elementary. We will first prove that \textit{(1)} implies \textit{(3)}. 
\vspace{5mm}

Let $(z- \lambda)h \in \mathcal{M}$ for some $h \in X$, and suppose that $h \notin \mathcal{M}$. The function $f:= (z-\lambda)h \in \mathcal{M}$ satisfies $f(\lambda)=0$ but $f \notin (z-\lambda)\mathcal{M}$. Hence the equivalence class $\bar{f} \in \mathcal{M}/(z-\lambda)\mathcal{M}$ is non zero. Since $\lambda \notin \mathcal{Z}(\mathcal{M})$ there exists some function $g \in \mathcal{M}$ such that $g(\lambda) \neq 0$, which also means that $g \notin (z-\lambda)\mathcal{M}$ so $\bar{g} \neq 0$. Since $\text{ind}(\mathcal{M})=1$ there exists $\mu \in \C \backslash \{ 0\}$ such that $\bar{g} = \mu \bar{f}$. That means precisely that $g  \in \mu f + (z-\lambda)\mathcal{M}$. By evaluating at $z=\lambda$ we get $g(0) = \mu f(\lambda) +0 = 0$ which is contradictory. 

\vspace{3mm}

To prove that \textit{(3)} implies \textit{(1)}, suppose that whenever $(z-\lambda)h \in \mathcal{M}$ for some $h \in X$, we have that $h \in \mathcal{M}$. Consider $f,g \in \mathcal{M}$ and their respective equivalence classes, $\bar{f},\bar{g} \in \mathcal{M}/(z-\lambda)\mathcal{M}$. We need to show that they are linearly dependent as vectors and that way conclude that the dimension of the quotient is in fact equal to one. If either $\bar{f}$ or $\bar{g}$ are zero then there is nothing to show, so we may assume neither of them are. That in particular means that, thanks to the hypothesis, that $f(\lambda),g(\lambda)\neq 0$. Consider the function $g_0(z) = g(z) \cdot \frac{f(\lambda)}{g(\lambda)} \in \mathcal{M}$. Then $f-g_0 \in \mathcal{M}$ and $f(\lambda) - g_0(\lambda) = 0$. Therefore we can write $f-g_0 = (z-\lambda)h$ for some $h \in X$, and by the hypothesis that means that $h \in \mathcal{M}$ and thus we conclude that $f-g_0 \in (z-\lambda)\mathcal{M}$. Therefore $\overline{f-g_0} = 0 \Ra \bar{f}-\bar{g_0} =0 \Ra \bar{f} = \bar{g_0} \Ra \bar{f} = \frac{f(\lambda)}{g(\lambda)} \bar{g} $, and thus the equivalence classes of $f$ and $g$ are linearly dependent.
\end{proof}

%%%%%%%%%%%%%%%%%%%%%%%%%%%%%%%%%%%%%%%%%%%
%
%
% END OF PROOF OF PROPOSITION 4.3
%
%
%%%%%%%%%%%%%%%%%%%%%%%%%%%%%%%%%%%%%%%%%%%

\begin{proposition}
Let $\mathcal{M} \in Lat_X(M_z,w^{\ast})$ and $\lambda \notin \mathcal{Z}(\mathcal{M})$. The following are equivalent:

\begin{enumerate}
\item $\text{ind}(\mathcal{M})=1$,
\item There exists a (not necessarily closed) subspace $L \sbs \mathcal{M}$ such that $\overline{L}^{w^{\ast}} = \mathcal{M}$, with the properties that $\lambda \notin \mathcal{Z}(L)$ and $(z-\lambda)h \in L $ for some $h \in X$ implies $h \in \mathcal{M}$.
\end{enumerate}
\end{proposition}

%%%%%%%%%%%%%%%%%%%%%%%%%%%%%%%%%%%%%%%%%%%
%
%
%  PROOF OF PROPOSITION 4.4
%
%
%%%%%%%%%%%%%%%%%%%%%%%%%%%%%%%%%%%%%%%%%%%

\begin{proof}
Proving that \textit{(1)} implies \textit{(2)} is achieved by simply taking $L = \mathcal{M}$ and applying Proposition 4.3. 

\vspace{3mm}

To prove the converse, we will verify condition \textit{(3)} of Proposition 4.3. Let $(z-\lambda)h \in \mathcal{M}$ for some $h \in X$. Since $\overline{L}^{w^{\ast}} = \mathcal{M}$, there exists a net $\{ f_i\}_{i\in I} \sbs L$ such that $f_i \stackrel{w^{\ast}}{\ra} (z-\lambda)h $. In particular by continuity of the evaluation functionals we have that $f_i(\lambda) = k_{\lambda}(f_i) \ra k_{\lambda}((z-\lambda)h) = 0$. Since $\lambda \notin \mathcal{Z}(L)$ we can find a $g \in L$ such that $g(\lambda)\neq 0$. For every $i \in I$ we consider the function 

\begin{equation*}
g_i(z) = f_i(z) - \frac{f_i(\lambda)}{g(\lambda)}g(z) \in L.
\end{equation*}

This new net $\{ g_i \}_{i\in I}$ has $g_i(\lambda)=0 $ for all $i \in I$. That means that there is a net $\{ h_i \}_{i \in I} \sbs X$ such that $g_i = (z-\lambda)h_i$. But by the hypothesis this means that every $h_i \in \mathcal{M}$ for all $i \in I$. Since $f_i \stackrel{w^{\ast}}{\ra}(z-\lambda)h$ we get that  $g_i \stackrel{w^{\ast}}{\ra}(z-\lambda)h$ and hence $(z-\lambda)h_i \stackrel{w^{\ast}}{\ra} (z-\lambda)h$. By continuity of the division operator we get $h_i \stackrel{w^{\ast}}{\ra}h$. Since $h_i \in \mathcal{M}$ for all $i \in I$, and $ \mathcal{M}$ is $w^{\ast}$-closed, we obtain that $h \in \mathcal{M}$. Condition \textit{(3)} is therefore satisfied and $\text{ind}(\mathcal{M})=1$.
\end{proof}

%%%%%%%%%%%%%%%%%%%%%%%%%%%%%%%%%%%%%%%%%%%
%
%
% END OF PROOF OF PROPOSITION 4.4
%
%
%%%%%%%%%%%%%%%%%%%%%%%%%%%%%%%%%%%%%%%%%%%

\begin{corollary}
Let $f \in X, f \neq 0$. Then $\text{ind}[f]_{\ast} = 1$.
\end{corollary}

%%%%%%%%%%%%%%%%%%%%%%%%%%%%%%%%%%%%%%%%%%%
%
%
% PROOF OF COROLLARY 4.5
%
%
%%%%%%%%%%%%%%%%%%%%%%%%%%%%%%%%%%%%%%%%%%%

\begin{proof}
Since $f \neq 0$ there is some $\lambda \in \D$ such that $f(\lambda) \neq 0$. It suffices then to verify condition (2) of Proposition 4.4 by taking $L = \{pf \, : \, p \,\,\text{polynomial} \}$.
\end{proof}

%%%%%%%%%%%%%%%%%%%%%%%%%%%%%%%%%%%%%%%%%%%
%
%
% END OF PROOF OF COROLLARY 4.5
%
%
%%%%%%%%%%%%%%%%%%%%%%%%%%%%%%%%%%%%%%%%%%%

\begin{theorem}
Let $\mathcal{M}_1 , \mathcal{M}_2 \in Lat_X(M_z,w^{\ast})$  have the index one property, and let $\lambda \notin \mathcal{Z}(\mathcal{M}_1) \cup\mathcal{Z}(\mathcal{M}_2) $. The following are equivalent:

\begin{enumerate}
\item $\text{ind}(\mathcal{M}_1 \vee \mathcal{M}_2)=1$,
\item There exist nets $\{g_i^1 \}_{i \in I} \sbs \mathcal{M}_1 $,$\{g_i^2 \}_{i \in I} \sbs \mathcal{M}_2 $ such that $g_i^1(\lambda) = g_i^2(\lambda) =1$ and $g_i^1 - g_i^2 \stackrel{w^{\ast}}{\ra} 0$.
\end{enumerate}
\end{theorem}

%%%%%%%%%%%%%%%%%%%%%%%%%%%%%%%%%%%%%%%%%%%
%
%
% PROOF OF THEOREM 4.6
%
%
%%%%%%%%%%%%%%%%%%%%%%%%%%%%%%%%%%%%%%%%%%%

\begin{proof}
Suppose that $\text{ind}(\mathcal{M}_1 \vee \mathcal{M}_2)=1$. Since $\lambda \notin \mathcal{Z}(\mathcal{M}_1) \cup\mathcal{Z}(\mathcal{M}_2)$ there are $f_1 \in \mathcal{M}_1 $, $f_2 \in \mathcal{M}_2 $ such that $f_1(\lambda) = f_2(\lambda) = 1$. We have that $f_1 - f_2 \in \mathcal{M}_1 \vee \mathcal{M}_2$ and $(f_1-f_2)(\lambda) = 0$ so $f_1 - f_2 = (z-\lambda)h$ for some $h \in X$. By Proposition 4.3, $h \in \mathcal{M}_1 \vee \mathcal{M}_2$ and since $\mathcal{M}_1 + \mathcal{M}_2$ is dense in $\mathcal{M}_1 \vee \mathcal{M}_2$ there exists a net $\{ h_i\}_{i \in I} \sbs \mathcal{M}_1 + \mathcal{M}_2$ with $h_i \stackrel{w^{\ast}}{\ra}h$. By the definition of $\mathcal{M}_1 + \mathcal{M}_2$ we can find nets $\{h_i^1 \}_{i \in I} \sbs \mathcal{M}_1 $,$\{h_i^2 \}_{i \in I} \sbs \mathcal{M}_2$ such that $h_i = h_i^1 - h_i^2 \stackrel{w^{\ast}}{\ra}h$. Then $(z-\lambda)h_i^1 - (z-\lambda)h_i^2\stackrel{w^{\ast}}{\ra} (z-\lambda)h = f_1 - f_2$ by continuity of the shift operator. Define the nets:

\begin{gather*}
g_i^1 = (z-\lambda)h_i^1 + f_1 \in \mathcal{M}_1 \\
g_i^2 = (z-\lambda)h_i^2 + f_2 \in \mathcal{M}_2 \,,\, \text{with} \, \\
g_i^1(\lambda) = 0 + f_1(\lambda) = 1 \,\, \text{and} \,\, g_i^2(\lambda) = 0 + f_2(\lambda) = 1
\end{gather*}

and notice that

\begin{gather*}
g_i^1 - g_i^2 = (z-\lambda)h_i^1 + f_1 -(z-\lambda)h_i^2 -f_2 = (z-\lambda)(h_i^1-h_i^2) +f_1 -f_2 \stackrel{w^{\ast}}{\ra} (z-\lambda)h - (z-\lambda)h = 0,
\end{gather*}

which gives (2). To prove the contrary we will verify condition (2) of proposition 4.4. To that end, take $L = \mathcal{M}_1 + \mathcal{M}_2 $ and consider a function $h \in X$ with $(z-\lambda)h \in \mathcal{M}_1 + \mathcal{M}_2 $. We will show that $h \in \mathcal{M}_1 \vee \mathcal{M}_2$. We write $(z-\lambda)h \in \mathcal{M}_1 + \mathcal{M}_2$ as $(z-\lambda)h = f_1 + f_2 $ with $f_i \in \mathcal{M}_i$. By the hypothesis there are nets 

\begin{gather*}
\{ g_i^1 \}_{i\in I} \sbs \mathcal{M}_1 , \{ g_i^2 \}_{i\in I} \sbs \mathcal{M}_2 \, ,\, \text{with} \\
g_i^1(\lambda) = g_i^2(\lambda) = 1 \,\, \text{and} \,\, g_i^1 - g_i^2 \stackrel{w^{\ast}}{\ra} 0.
\end{gather*}

We write:

\begin{gather*}
f_1(z) = f_1(z) - f_1(\lambda)g_i^1(z) +  f_1(\lambda)g_i^1(z), \\
f_2(z) = f_2(z) - f_2(\lambda)g_i^2(z) +  f_2(\lambda)g_i^2(z).
\end{gather*}

and notice that:

\begin{gather*}
f_1 - f_1(\lambda)g_i^1 \in \mathcal{M}_1 \,\, \text{and}\,\, f_1(\lambda) - f_1(\lambda)g_i^1(\lambda) = 0 \,\, \fa i \in I , \\
f_2 - f_2(\lambda)g_i^2 \in \mathcal{M}_2 \,\, \text{and}\,\, f_2(\lambda) - f_2(\lambda)g_i^2(\lambda) = 0 \,\, \fa i \in I.
\end{gather*}

Then there are nets $\{ h_i^1 \}_{i\in I},\{ h_i^2 \}_{i\in I} \sbs X$ such that: 

\begin{gather*}
f_1 - f_1(\lambda)g_i^1 = (z-\lambda)h_i^1 \,\, \text{and} \,\, f_2 - f_2(\lambda)g_2^1 = (z-\lambda)h_i^2.
\end{gather*}

Moreover, by the fact that each of the subspaces is of index one and by Proposition 4.3 we can conclude that in fact $h_i^1 \in \mathcal{M}_1$ and $h_i^2 \in \mathcal{M}_2$ for all $i \in I$. Notice as well that $f_1(\lambda)+f_2(\lambda) = 0 \Ra f_1(\lambda) = -f_2(\lambda)$. This permits us to write:

\begin{gather*}
f_1 + f_2 = (z-\lambda)h_i^1 + f_1(\lambda)g_i^1 + (z-\lambda)h_i^2 + f_2(\lambda)g_i^2 = (z-\lambda)(h_i^1 + h_i^2) + f_1(\lambda)(g_i^1-g_i^2).
\end{gather*}

Since $g_i^1 - g_i^2 \stackrel{w^{\ast}}{\ra} 0$ we have that $(z-\lambda)(h_i^1+h_i^2) \stackrel{w^{\ast}}{\ra} f_1+f_2 = (z-\lambda)h $. Once again by continuity of the division operator we may conclude that $(h_i^1+h_i^2) \stackrel{w^{\ast}}{\ra} h$, which gives that $h \in \mathcal{M}_1 \vee \mathcal{M}_2 $.

\end{proof}

%%%%%%%%%%%%%%%%%%%%%%%%%%%%%%%%%%%%%%%%%%%
%
%
% END OF PROOF OF THEOREM 4.6
%
%
%%%%%%%%%%%%%%%%%%%%%%%%%%%%%%%%%%%%%%%%%%%

\begin{proposition}
Let $\mathcal{M} \in Lat_{X_0}(M_z)$ be an invariant subspace of index 1. Then $\overline{\mathcal{M}}^{w^{\ast}} \in Lat_X(M_z, w^{\ast})$ has index 1.
\end{proposition}

%%%%%%%%%%%%%%%%%%%%%%%%%%%%%%%%%%%%%%%%%%%
%
%
% PROOF OF PROPOSITION 4.7
%
%
%%%%%%%%%%%%%%%%%%%%%%%%%%%%%%%%%%%%%%%%%%%

\begin{proof}
Let $\lambda \notin \mathcal{Z}(\mathcal{M})$. Then $\lambda \notin \mathcal{Z}(\overline{\mathcal{M}}^{w^{\ast}})$. We set $L = \mathcal{M}$ and then $\overline{L}^{w^{\ast}} = \overline{\mathcal{M}}^{w^{\ast}}$. Let $h \in X$ with $(z - \lambda)h \in L = \mathcal{M} \sbs X_0$. Since $\mathcal{M}$ has index 1, we deduce from the analog of Proposition 4.3 for the norm topology case that $h \in \mathcal{M} $. By Proposition 4.4, $\text{ind}(\overline{\mathcal{M}}^{w^{\ast}}) = 1$.
\end{proof}

%%%%%%%%%%%%%%%%%%%%%%%%%%%%%%%%%%%%%%%%%%%
%
%
% END OF PROOF OF PROPOSITION 4.7
%
%
%%%%%%%%%%%%%%%%%%%%%%%%%%%%%%%%%%%%%%%%%%%

\begin{proposition}
Let $\mathcal{M} \in Lat_{X_0}(M_z)$ be an invariant subspace with $\text{ind}(\mathcal{M}) = M $, where $M$ is finite or countably infinite. Then $\overline{\mathcal{M}}^{w^{\ast}} \in Lat_X(M_z, w^{\ast})$ satisfies $\text{ind}(\overline{\mathcal{M}}^{w^{\ast}}) \leq M$. 
\end{proposition}

%%%%%%%%%%%%%%%%%%%%%%%%%%%%%%%%%%%%%%%%%%%
%
%
% PROOF OF PROPOSITION 4.8
%
%
%%%%%%%%%%%%%%%%%%%%%%%%%%%%%%%%%%%%%%%%%%%

\begin{proof}
By (1) of Proposition 4.2, we may write $\mathcal{M} = \mathcal{M}_1 \vee \mathcal{M}_2 \vee \cdots \vee \mathcal{M}_M$, where each $\mathcal{M}_n$ has index 1. Then $\overline{\mathcal{M}}^{w^{\ast}} = \overline{\mathcal{M}_1}^{w^{\ast}} \vee \overline{\mathcal{M}_2}^{w^{\ast}} \vee \cdots \vee \overline{\mathcal{M}_M}^{w^{\ast}}$. By combining Propositions 4.2 and 4.7 we obtain the result.
\end{proof}

%%%%%%%%%%%%%%%%%%%%%%%%%%%%%%%%%%%%%%%%%%%
%
%
% END OF PROOF OF PROPOSITION 4.8
%
%
%%%%%%%%%%%%%%%%%%%%%%%%%%%%%%%%%%%%%%%%%%%

\begin{lemma}
    If $\mathcal{M} \sbs X_0$, is a norm closed and convex set, then $\overline{\mathcal{M}}^{w^{\ast}} \cap X_0 = \mathcal{M}$.
\end{lemma}

\begin{proof}
    One inclusion is obvious. For the other inclusion consider $h \in \overline{\mathcal{M}}^{w^{\ast}} \cap X_0  $. There exists a net $\{h_i \}_{i\in I} \sbs \mathcal{M} \sbs X_0$ such that $h_i \stackrel{w^{\ast}}{\ra} h$ in $X$. Notice that both the net and its limit belong to the space $X_0$. By the dualities $X_0 \cong Y$ and $Y \cong X$ we know that $h_i \stackrel{w^{\ast}}{\ra} h$ in $X$ is the same as $h_i \ra h$ weakly in $X_0$. Hence $h$ belongs to the weak closure of $\mathcal{M}$ in $X_0$. Since $\mathcal{M}$ is convex, its weak closure coincides with its norm closure, and as such we deduce that $h \in \mathcal{M}$, as $\mathcal{M}$ is itself norm closed.
\end{proof}

We may now prove one of the main theorems of this section.

%%%%%%%%%%%%%%%%%%%%%%%%%%%%%%%%%%%%%%%%%%%
%
%
% PROOF OF THEOREM 1.6
%
%
%%%%%%%%%%%%%%%%%%%%%%%%%%%%%%%%%%%%%%%%%%%

\vspace{2mm}

\noindent \textit{Proof of Theorem 1.6.}

\vspace{2mm}

We assume first that $M< \infty$ and suppose that $\text{ind}(\overline{\mathcal{M}}^{w^{\ast}}) < M$. Without loss of generality we may assume that $0 \notin \mathcal{Z}(\mathcal{M})$.  Let $\{f_1,f_2,\ldots,f_M \}$ be a set of functions whose equivalence classes in $\mathcal{M}/z\mathcal{M}$ form a base. By our assumption, the set $\{ \overline{f}_1, \overline{f}_2, \ldots, \overline{f}_M \} \sbs \overline{\mathcal{M}}^{w^{\ast}}/ z\overline{\mathcal{M}}^{w^{\ast}}$ has to be linearly dependent, and therefore there exist $\lambda_1,\lambda_2, \ldots, \lambda_M$ not all zero such that:

\begin{equation*}
    \sum_{m=1}^M\lambda_m \overline{f}_m = 0 \, , \,\, \text{in} \,\, \overline{\mathcal{M}}^{w^{\ast}}/ z\overline{\mathcal{M}}^{w^{\ast}}.
\end{equation*}

That means that there exists a function $h \in \overline{\mathcal{M}}^{w^{\ast}}$ such that:

\begin{equation}
    \sum_{m=1}^M\lambda_m f_m =  zh \, , \,\, \text{in} \, \, X.
\end{equation}

It suffices to prove that $h \in \mathcal{M}$, because that provides a contradiction to the fact that the equivalence classes of the functions $f_m  , 1 \leq m \leq M$ form a base in  $\mathcal{M}/z \mathcal{M}$. Notice that (27) actually says that $zh \in X_0$. From the division property on $X_0$ we deduce that $h \in X_0$. Since $h \in \overline{\mathcal{M}}^{w^{\ast}}$ we get that $h \in \overline{\mathcal{M}}^{w^{\ast}} \cap X_0$. Moreover, $\mathcal{M}$ is convex, as it is a linear subspace of $X_0$. From Lemma 4.9, we obtain that $h \in \mathcal{M}$. 

In the above argument we essentially demonstrated that, given an $\mathcal{M} \in Lat_{X_0}(M_z)$ and a set of functions $f_1,f_2,\ldots,f_M$, linearly independent in the quotient space $\mathcal{M}/z\mathcal{M}$, the same set of functions forms a linearly independent set in the quotient space $\overline{\mathcal{M}}^{w^{\ast}}/ z\overline{\mathcal{M}}^{w^{\ast}}$. In the case where $M = \infty $ we consider an infinite set $\{ f_1, f_2, \ldots \} \sbs \mathcal{M}$ that forms a base in $\mathcal{M}/z\mathcal{M}$. Then for any $N \in \N$ the set $\{f_1, f_2, \ldots f_N \}$ will form a linearly independent set in $\overline{\mathcal{M}}^{w^{\ast}}/ z\overline{\mathcal{M}}^{w^{\ast}}$, proving that for every $N \in \N$, $\text{ind}(\overline{\mathcal{M}}^{w^{\ast}}) \geq N $, and hence $\text{ind}(\overline{\mathcal{M}}^{w^{\ast}}) = \infty$.

$ \quad\quad\quad\quad\quad\quad\quad\quad\quad\quad\quad\quad\quad\quad\quad\quad\quad\quad\quad\quad\quad\quad\quad\quad\quad\quad\quad\quad\quad\quad\quad\quad\quad\quad\quad\quad\quad\quad\quad\quad \quad \blacksquare$

%%%%%%%%%%%%%%%%%%%%%%%%%%%%%%%%%%%%%%%%%%%
%
%
% END OF PROOF OF THEOREM 1.6
%
%
%%%%%%%%%%%%%%%%%%%%%%%%%%%%%%%%%%%%%%%%%%%

Note that the above argument works exactly in the same way if $M=1$, reproving Proposition 4.7. Lastly, we can apply this to prove Theorem 1.7:

\vspace{2mm}

\noindent \textit{Proof of Theorem 1.7.}

\vspace{2mm}

Let $N \in \{2, \ldots, \infty \}$. By Theorem 1.5 there exist functions $f_n \sbs \BB_0$, $1 \leq n < N$ such that $E_N := \vee_{n=1}^{N}[f_n]$ has index $N$ in $\BB_0$. An application of Theorem 1.6 for for $X_0 = \BB_0$, $Y= A^1$ and $X = \BB$ yields the result.

$ \quad\quad\quad\quad\quad\quad\quad\quad\quad\quad\quad\quad\quad\quad\quad\quad\quad\quad\quad\quad\quad\quad\quad\quad\quad\quad\quad\quad\quad\quad\quad\quad\quad\quad\quad\quad\quad\quad\quad\quad \quad \blacksquare $

%%%%%%%%%%%%%%%%%%%%%%%%%%%%%%%%%%%%%%%%%%%
%
%
%END OF PROOF OF COROLLARY 1.7
%
%
%%%%%%%%%%%%%%%%%%%%%%%%%%%%%%%%%%%%%%%%%%%

\begin{scriptsize}
\noindent \bf{ACKNOWLEDGEMENTS.} 
\end{scriptsize}
I would like to thank my advisors Evgeny Abakumov and Alexander Borichev for proposing me the problems and for proof-checking my drafts. I would especially like to thank A. Borichev for sharing with me his rich ideas and valuable techniques.

%-------------------------------------------------------------------------
%
% Bibliography
%
%-------------------------------------------------------------------------

\printbibliography[heading=bibintoc,title={References}]

%-------------------------------------------------------------------------
%
% Contact Info
%
%-------------------------------------------------------------------------

\vspace{8mm}

\noindent NIKIFOROS BIEHLER: 

Univ Gustave Eiffel, Univ Paris Est Creteil, CNRS, LAMA UMR8050, F-77447 Marne-la-Vallée, France
 
 Email address : \bf{nikiforos.biehler@univ-eiffel.fr}

\end{document}